\documentclass[final,leqno]{siamltex704}
\usepackage{amsmath}
\usepackage{graphicx}
\usepackage[notcite,notref]{showkeys}
\usepackage{mathrsfs}
\usepackage{float}
\usepackage{amsfonts,amssymb}
\usepackage{dsfont}
\usepackage{pifont}
\usepackage{hyperref}
\usepackage{multirow}
\usepackage{color}

\renewcommand{\ldots}{\dotsc}
\newtheorem{algorithm}{Weak Galerkin Algorithm}

\newcommand{\bu}{{\bf u}}

\newcommand{\bw}{{\bf w}}

\newcommand{\bx}{{\bf x}}

\newcommand{\be}{{\bf e}}
\newcommand{\bv}{{\mathbf v}}

\def\T{{\mathcal T}}

\def\pT{{\partial T}}

\def\e{\varepsilon}

\def\bbf{{\bf f}}

\def\bn{{\bf n}}

\def\3bar{{|\hspace{-.02in}|\hspace{-.02in}|}}

\def\bbR{\mathbb{R}}
\newcommand{\bm}[1]{\mbox{\boldmath{$#1$}}}
\def\bPhi{\bm\Phi}

\title {A Locking-Free Weak Galerkin Finite Element Method for Elasticity Problems
in the Primal Formulation}

\author{
Chunmei Wang\thanks{School of Mathematics, Georgia Institute of
Technology, Atlanta, Georgia, 30332; Taizhou College, Nanjing Normal
University, Taizhou, China 225300. This research was partially
supported by National Science Foundation award \#DMS-1522586.} \and
Junping Wang\thanks{Division of Mathematical Sciences, National
Science Foundation, Arlington, VA 22230 (jwang@nsf.gov). The
research of Junping Wang was supported by the NSF IR/D program,
while working at National Science Foundation. However, any opinion,
finding, and conclusions or recommendations expressed in this
material are those of the author and do not necessarily reflect the
views of the National Science Foundation.} \and Ruishu
Wang\thanks{Department of Mathematics, Jilin University, Changchun
China.} \and Ran Zhang\thanks{Department of Mathematics, Jilin
University, Changchun China. This research was supported in part by
China Natural National Science Foundation (11271157, 11371171,
11471141), and by the Program of China Ministry of Education for New
Century Excellent Talents in Universities.}}
\begin{document}

\maketitle

\begin{abstract}
This paper presents an arbitrary order locking-free numerical scheme
for linear elasticity on general polygonal/polyhedral partitions by
using weak Galerkin (WG) finite element methods. Like other WG
methods, the key idea for the linear elasticity is to introduce
discrete weak strain and stress tensors which are defined and
computed by solving inexpensive local problems on each element. Such
local problems are derived from weak formulations of the
corresponding differential operators through integration by parts.
Locking-free error estimates of optimal order are derived in a
discrete $H^1$-norm and the usual $L^2$-norm for the approximate
displacement when the exact solution is smooth. Numerical results
are presented to demonstrate the efficiency, accuracy, and the
locking-free property of the weak Galerkin finite element method.
\end{abstract}



\pagestyle{myheadings}

\section{Introduction}
In this paper, we are concerned with the development of efficient
new numerical methods for linear elasticity equations by using the
weak Galerkin finite element method recently developed in
\cite{wy1202, wysec2, wy1302, mwy2}. Let $\Omega\subset\mathbb R^d\
(d = 2, 3)$ be an open bounded and connected domain with Lipschitz
continuous boundary $\Gamma=\partial\Omega$ of an elastic body
subject to an exterior force $\bbf$ and a given displacement
boundary condition. The kinematic model of linear elasticity seeks a
displacement vector field $\bu$ satisfying
\begin{eqnarray}\label{primal_model}
-\nabla\cdot\sigma(\bu)&=&\bbf,\qquad\text{in}\ \Omega,\\
\bu&=&\widehat{\bu},\qquad\text{on}\ \Gamma, \label{bc1}
\end{eqnarray}
where $\sigma(\bu)$ is the symmetric Cauchy stress tensor. For
linear, homogeneous, and isotropic materials, the Cauchy stress
tensor is given by
$$
\sigma(\bu)=2\mu\varepsilon(\bu)+\lambda(\nabla\cdot\bu)\textbf{I},
$$
where $\varepsilon(\bu)=\frac{1}{2}(\nabla \bu+\nabla \bu^T)$ is the
linear strain tensor, $\mu$ and $\lambda$ are the Lam\'{e}
constants. For linear plane strain, the Lam\'{e} constants are given
by
$$
\lambda= \frac{E\nu}{(1+\nu)(1-2\nu)},\qquad \mu=\frac{E}{2(1+\nu)},
$$
where $E$ is the elasticity modulus and $\nu$ is Poisson's ratio.

Denote by $(\cdot,\cdot)$ the $L^2$-inner product in either
$L^2(\Omega)$, $[L^2(\Omega)]^{d}$, or $[L^2(\Omega)]^{d\times d}$,
as appropriate. A weak formulation for (\ref{primal_model}) in the
primal form reads as follows: Find $\bu\in [H^1(\Omega)]^d$
satisfying $\bu=\widehat\bu$ on $\Gamma$ and
\begin{eqnarray}\label{model-weak}
2(\mu\varepsilon(\bu),\varepsilon(\bv))+(\lambda\nabla
\cdot\bu,\nabla \cdot\bv)&=&(\bbf,\bv), \qquad \forall \bv\in
[H_0^1(\Omega)]^d,
\end{eqnarray}
where $H^1(\Omega)$ is the usual Sobolev space defined by
$$
H^1(\Omega)=\{v:\; v\in L^2(\Omega),\ \nabla v\in L^2(\Omega)\},
$$
and $H_0^1(\Omega)$ is the closed subspace of $H^1(\Omega)$
consisting of all the functions with vanishing boundary value.

The main objective of this paper is to study an application of the
weak Galerkin finite element method \cite{wy1202, wysec2, wy1302,
ww-survey, mwy, wwbi} to the linear elasticity problem
(\ref{primal_model})-(\ref{bc1}) based on the primal formulation
(\ref{model-weak}). The weak Galerkin (WG) refers to a generic
finite element technique for partial differential equations where
differential operators are approximated or reconstructed by solving
inexpensive local problems on each element. Such local problems are
often derived from weak formulations of the corresponding
differential operators through integration by parts. Recent work on
WG has revealed that the concept of discrete weak derivatives offers
a new paradigm in the discretiztion of partial differential
equations. The resulting numerical schemes often possess a great
robustness in stability and convergence for which other competing
methods are hard to achieve. For the linear elasticity problem
(\ref{model-weak}), we shall demonstrate that the WG numerical
approximations are not only accurate and robust with respect to the
polygonal/polyhedral partition of the domain, but also naturally
``locking-free" in terms of the Lam\'{e} constant $\lambda$. This is
a result that the standard conforming finite element method does not
have.

``Locking" refers to a phenomenon of numerical approximations for a
certain problems whose mathematical formulations involve a parameter
dependency. For the linear elasticity problem, the parameter is the
Poisson ratio $\nu$. For $\nu$ close to $\frac12$ (i.e., when the
material is nearly incompressible), it is well known that various
finite element schemes, such as the continuous piecewise linear
elements, results in poor observed convergence rates in the
displacements. In 1983, Vogelious \cite{vogelius} showed absence of
locking for the $p$-version of the finite element method on smooth
domains. Later on, Scott and Vegelious \cite{scott-vogelius} proved
that no locking results when polynomials of degree $k\ge 4$ are used
on triangular meshes. However, Babu\v{s}ka and Suri
\cite{babuska-suri} found, for conforming methods, locking cannot be
avoided on quadrilateral meshes for any polynomial of degree $k\ge
1$. In the discontinuous Galerkin context, Hansbo and Larson
\cite{larson} proved that the numerical approximation arising from a
discontinuous Galerkin method is locking-free for any values of
$k\ge 1$, also see \cite{Wihler} for the case of $k=1$. In
\cite{Daniele.01}, Daniele and Nicaise designed a locking-free
discontinuous Galerkin method for composite materials featuring
quasi-incompressible and compressible sections.

Locking occurs because for the limit case $\nu=\frac12$ (or
$\lambda=\infty$), the exact solution of the displacement must
satisfy the constraint $\nabla\cdot\bu=0$. But locking is not a
difficult issue to resolve for the linear elasticity equation. One
possibility is the use of mixed methods by reformulating the linear
elasticity equation as a generalized Stokes equation (see
\cite{Brezzi-Fortin, Girault, abd1984} and the references cited
therein). Specifically, by introducing a pressure variable
$p=\lambda\nabla\cdot\bu$, the elasticity problem
(\ref{primal_model})-(\ref{bc1}) can be reformulated as follows:
Find $\bu\in [H^1(\Omega)]^d$ and $p\in L^2(\Omega)$ satisfying
$\bu=\widehat\bu$ on $\Gamma$, the compatibility condition
$\int_\Omega \lambda^{-1} p dx = \int_\Gamma \widehat\bu\cdot\bn
ds$, and the following equations
\begin{align}\label{var1}
2( \mu\varepsilon(\bu),\varepsilon(\bv))+(\nabla
\cdot\bv,p)&=(\bbf,\bv), \qquad \forall
\bv\in [H_0^1(\Omega)]^d,\\
(\nabla\cdot \bu,q)-(\lambda^{-1}p,q) &=0,\qquad\qquad \forall q\in
L^2_0(\Omega). \label{var2}
\end{align}
Here $\bn$ is the outward normal direction on $\Gamma$, and
$L_0^2(\Omega)$ is the closed subspace of $L^2(\Omega)$ consisting
of all functions with mean-value zero. Consequently, any finite
elements that are stable for the Stokes problem would provide a
locking-free approximation for the linear elasticity problem
(\ref{primal_model})-(\ref{bc1}), but at the cost of solving a
saddle-point problem with an additional pressure variable.

The weak Galerkin method can also be applied to the linear
elasticity problem based on the mixed formulation
(\ref{var1})-(\ref{var2}). In principle, such applications should
yield locking-free numerical approximations for the displacement
variable. Surprisingly, we found that the weak Galerkin finite
element method based on the primal formulation (\ref{model-weak}) is
equivalent to the weak Galerkin when applied to the mixed
formulation (\ref{var1})-(\ref{var2}). This equivalence indicates
that the weak Galerkin approximation arising from the primal
formulation (\ref{model-weak}) might be locking-free. The main goal
of this paper is to provide a rigorous mathematical justification to
the locking-free nature of weak Galerkin. Specifically, we shall
perform the following tasks in this study: (1) propose two weak
Galerkin finite element schemes, one based on the primal formulation
(\ref{model-weak}) and the other based on the mixed formulation
(\ref{var1})-(\ref{var2}); (2) show that the two weak Galerkin
finite element schemes are equivalent; (3) establish a locking-free
convergence theory for the WG scheme based on
(\ref{var1})-(\ref{var2}); and (4) numerically demonstrate the
accuracy and the locking-free property of the proposed weak Galerkin
finite element methods. The main mathematical challenge of this
research lies in the error estimate for shape-regular finite element
partitions consisting of arbitrary polygon or polyhedra. The
corresponding technicality is represented by the inequality
(\ref{EQ:July29:815}) in Lemma \ref{korninqua}, which is a result of
a creative use of the Korn's inequality and the domain inverse
inequality \cite{wy1202}.

Throughout the paper, we will follow the usual notation for Sobolev
spaces and norms \cite{ciarlet-fem}. For any open bounded domain
$D\subset \mathbb{R}^d$ with Lipschitz continuous boundary, we use
$\|\cdot\|_{s,D}$ and $|\cdot|_{s,D}$ to denote the norm and
seminorms in the Sobolev space $H^s(D)$ for any $s\ge 0$,
respectively. The inner product in $H^s(D)$ is denoted by
$(\cdot,\cdot)_{s,D}$. The space $H^0(D)$ coincides with $L^2(D)$,
for which the norm and the inner product are denoted by $\|\cdot
\|_{D}$ and $(\cdot,\cdot)_{D}$, respectively. When $D=\Omega$, we
shall drop the subscript $D$ in the norm and inner product notation.

The paper is organized as follows. Section 2 is devoted to a
discussion of weak divergence and weak gradient for vector-valued
functions as well as their discrete analogues. In Section 3, we
present a weak Galerkin finite element method for the linear
elasticity problem based on the primal formulation
(\ref{model-weak}). In Section 4, we describe another weak Galerkin
finite element method based on the mixed formulation
(\ref{var1})-(\ref{var2}). It is also shown in Section 4 that the
two weak Galerkin methods are equivalent. Section 5 is devoted to a
discussion of stability conditions (i.e., the {\em inf-sup}
condition and coercivity estimates) for the mixed weak Galerkin
finite element method. In Section 6, we prepare ourselves for error
estimates by deriving an identity. Section 7 is devoted to the
establishment of an optimal order error estimate in a discrete
$H^1$-norm. In Section 8, we use the usual duality argument to
derive an optimal order error estimate in the $L^2$-norm for the
displacement variable. In Section 9, we derive some supporting tools
and inequalities useful for error analysis. Finally, in Section 10,
we report some numerical results that demonstrate the accuracy and
locking-free nature of the proposed weak Galerkin finite element
methods for the linear elasticity problem.

\section{Weak Divergence and Weak Gradient Operators}
The divergence and gradient operators are two primary differential
operators in the variational problem (\ref{var1})-(\ref{var2}). The
goal of this section is to review the definition and computation of
weak gradient and weak divergence operators which have been
introduced and studied in applications to other partial differential
equations \cite{wysec2, wy1202, wy1302}.

Let $K$ be a polygon in 2D or polyhedra in 3D with boundary
$\partial K$. Define the space of weak vector-valued functions in
$K$ as follows
$$
V(K)=\{\bv=\{\bv_0,\bv_b\}:\; \bv_0\in [L^2(K)]^d,\bv_b \in
[L^2(\partial K)]^d\},
$$
where $\bv_0$ and $\bv_b$ represent the values of $\bv$ in $K$ and
on the boundary $\partial K$, respectively. Note that $\bv_b$ is not
necessarily related to the trace of $\bv_0$ should it be
well-defined. Denote by $\langle\cdot,\cdot\rangle_{\partial K}$ the
standard inner-product in either $L^2(\partial K)$ or $[L^2(\partial
K)]^d$, as appropriate.

\begin{definition} \cite{wy1202, wy1302}(weak divergence)
The weak divergence of $\bv\in V(K)$, denoted by $\nabla_w\cdot
\bv$, is a bounded linear functional in the Sobolev space $H^1(K)$,
so that its action on any $\phi\in H^1(K)$ is given by
\begin{equation*}
\langle\nabla_w\cdot \bv,\phi\rangle_K:=-(\bv_0,\nabla
\phi)_K+\langle \bv_b\cdot \bn,\phi\rangle_{\partial
K},
\end{equation*}
where $\bn$ is the unit outward normal direction on $\partial K$.
\end{definition}

For any non-negative integer $r$, denote by $P_r(K)$ the set of
polynomials with degree $r$ or less on $K$.

\begin{definition} \cite{wy1202, wy1302}(discrete weak divergence)
The discrete weak divergence  of $\bv\in V(K)$, denoted by
$\nabla_{w,r,K}\cdot \bv$, is the unique polynomial in
$P_r(K)$, satisfying
\begin{equation}\label{2.2}
(\nabla_{w,r,K}\cdot \bv,\phi)_K=-(\bv_0,\nabla
\phi)_K+\langle \bv_b\cdot \bn,
\phi\rangle_{\partial K},\qquad \forall\phi\in P_r(K),
\end{equation}
where $\bn$ is the unit outward normal direction on $\partial K$.
\end{definition}

\begin{definition} \cite{wysec2, wy1302}(weak gradient)
The weak gradient of $\bv\in V(K)$, denoted by $\nabla_w \bv$, is a
matrix-valued bounded linear functional in the Sobolev space
$[H^1(K)]^{d\times d}$, so that its action on any $\varphi\in
[H^1(K)]^{d\times d}$ is given by
\begin{equation*}
\langle\nabla_w  \bv,\varphi\rangle_K:=-(\bv_0,\nabla \cdot
\varphi)_K+\langle \bv_b,  \varphi\bn \rangle_{\partial K},
\end{equation*}
where $\bn$ is the unit outward normal direction on $\partial K$.
Here the divergence $\nabla\cdot\varphi$ is applied to each row of
$\varphi$, and $\varphi\bn$ is the usual matrix-vector
multiplication.
\end{definition}

\begin{definition} \cite{wysec2, wy1302}(discrete weak gradient)
The discrete weak gradient of $\bv\in V(K)$, denoted by
$\nabla_{w,r,K} \bv$, is the unique matrix-valued polynomial in
$[P_r(K)]^{d\times d}$, satisfying
\begin{equation}\label{2.5}
(\nabla_{w,r,K} \bv,\varphi)_K=-(\bv_0,\nabla \cdot
\varphi)_K+\langle \bv_b,  \varphi\bn
\rangle_{\partial K},\quad \forall \varphi\in[P_r(K)]^{d\times d},
\end{equation}
where $\bn$ is the unit outward normal direction on $\partial K$.
\end{definition}

Using the weak gradient, we may define the weak strain tensor as
follows
\begin{equation}\label{weak-strain}
\varepsilon_w(\bu)=\frac{1}{2}(\nabla_w\bu+\nabla_w\bu^T).
\end{equation}
Analogously, the weak stress tensor can be defined by
\begin{equation}\label{weak-stress}
\sigma_w(\bu)=2\mu\varepsilon_w(\bu) + \lambda
(\nabla_w\cdot\bu)\textbf{I}.
\end{equation}

\section{Numerical Algorithms}\label{Section:fem-algorithms}
Let ${\cal T}_h$ be a finite element partition of the domain
$\Omega\subset \mathbb{R}^d$ consisting of polygons in 2D or
polyhedra in 3D which are shape regular interpreted as in
\cite{wy1202}. For each $T\in \T_h$, denote by $h_T$ the diameter of
$T$. The mesh size of $\T_h$ is defined as $h=\max_{T\in \T_h} h_T$.
Let $T\in {\cal T}_h$ be an element with $e$ as an edge in 2D or a
face in 3D. Denote by ${\cal E}_h$ the set of all edges or faces in
${\cal T}_h$ and ${\cal E}^0_h={\cal E}_h\setminus {\partial\Omega}$
the set of all interior edges or faces in ${\cal T}_h$.

On each element $T\in\T_h$, denote by ${\textsf{RM}}(T)$ the space
of rigid motions on $T$ given by
$$
\textsf{RM}(T)=\{{\bf{a}}+\eta{\bf{x}}:\; {\bf{a}}\in \mathbb{R}^d,
\ \eta \in so(d)\},
$$
where ${\bf{x}}$ is the position vector on $T$ and $so(d)$ is the
space of skew-symmetric $d\times d$ matrices. The trace of the rigid
motion on each edge $e\subset T$ forms a finite dimensional space
denoted by
$$
P_{RM}(e)=\{\bv\in [L^2(e)]^d: \ \bv=\tilde\bv|_e \mbox{ for some }
\tilde{\bv}\in {\sf{RM}}(T),\ e\subset \pT\}.
$$

For any positive integer $k\ge 1$ and element $T\in\T_h$, we
introduce a local weak finite element space as follows
\begin{equation}\label{EQ:WFE-local}
 V(k,T)= \left\{ \bv=\{\bv_0,\bv_b\}:\ \bv_0\in [P_k(T)]^d,\
 \bv_b\in V_{k-1}(e)
 \right\},
\end{equation}
where $V_{k-1}(e)\equiv [P_{k-1}(e)]^d+P_{RM}(e)$. Since
$P_{RM}(e)\subset P_1(e)$, then the boundary component $V_{k-1}(e)$
is given by $[P_{k-1}(e)]^d$ for $k>1$ and $P_{RM}(e)$ for $k=1$.

The global weak finite element space $V_h$ is given by patching the
local spaces $V(k,T)$ with a single-valued component $\bv_b$ on each
interior element interface. All the weak finite element functions
$\bv\in V_h$ with vanishing boundary value $\bv_b=0$ on $\Gamma$
form a subspace of $V_h$, which is denoted as
\begin{equation}\label{EQ:WFE-global}
V^0_h=\{\bv =\{\bv _0,\bv_b\}\in V_h:\ \bv_b=0 ~\text{on}~ \Gamma\}.
\end{equation}

For each $\bv\in V_h$, the discrete weak divergence
$\nabla_{w,k-1}\cdot\bv$ and the discrete weak gradient
$\nabla_{w,k-1}\bv$ are computed by using (\ref{2.2}) and
(\ref{2.5}) on each element $T\in {\cal T}_h$; i.e.,
\begin{align*}
(\nabla_{w,k-1}\cdot \bv)|_T=&\nabla_{w,k-1,T}\cdot
(\bv|_T), \qquad \bv\in V_h,\\
(\nabla_{w,k-1}  \bv)|_T=&\nabla_{w,k-1,T} (\bv|_T), \qquad\quad
\bv\in V_h.
\end{align*}
For simplicity of notation, we shall drop the subscript $k-1$ from
the notation $\nabla_{w,k-1}$ and $\nabla_{w,k-1}\cdot $ in the rest
of the paper. For each edge or face $e\in\mathcal{E}_h$, denote by
$Q_b$ the $L^2$ projection operator onto the space $V_{k-1}(e)$;
i.e., the polynomial space $[P_{k-1}(e)]^2$ for $k>1$ or the rigid
motion space $P_{RM}(e)$ for $k=1$.

Next, we introduce two bilinear forms
\begin{eqnarray}\label{EQ:stabilizer}
s(\bw,\bv)&=&\sum_{T\in\mathcal{T}_h}h_T^{-1}\langle Q_b\bw_0-\bw_b,Q_b\bv_0-\bv_b\rangle_{\partial T},
\\
a_s(\bw,\bv)&=&\sum_{T\in\mathcal{T}_h}2(\mu\varepsilon_w(\bw),\varepsilon_w(\bv))_T
+\sum_{T\in\mathcal{T}_h}(\lambda\nabla_w\cdot\bw,\nabla_w\cdot\bv)_T+s(\bw,\bv),\label{EQ:bilinearForm}
\end{eqnarray}
where $\varepsilon_w(\bw)$ and $\nabla_w\cdot\bw$ are computed by
using the discrete weak gradient and weak divergence operators.

\begin{algorithm}\label{algo-primal}
For a numerical solution of the elasticity problem
(\ref{model-weak}), find $\bu_h=\{\bu_0, \bu_b\}\in V_h$ with $\bu_b
= Q_b \hat{\bu}$ on $ \Gamma$ such that
\begin{eqnarray}\label{WGA_primal}
a_s(\bu_h,\bv)=(\bbf,\bv_0), \qquad\forall \bv=\{\bv_0, \bv_b\}\in
V_h^0.
\end{eqnarray}
\end{algorithm}

The rest of this section is devoted to a study of (\ref{WGA_primal})
on the solution existence and uniqueness.

\begin{theorem}
There exists one and only one solution to the weak Galerkin finite
element scheme (\ref{WGA_primal}).
\end{theorem}

\begin{proof}
Since the number of equations equals the number of unknowns in
(\ref{WGA_primal}), it suffices to prove the solution uniqueness. To
this end, let $\bu_h^{(j)}=\{\bu_0^{(j)}, \bu_b^{(j)}\}\in V_h,\
j=1,2$, be two solutions of (\ref{WGA_primal}). It follows that
$\bu_b^{(j)} = Q_b \hat{\bu}$ on $ \Gamma$ and
\begin{eqnarray*}
a_s(\bu_h^{(j)},\bv)=(\bbf,\bv_0), \qquad\forall \bv=\{\bv_0,
\bv_b\}\in V_h^0, \ j=1,2.
\end{eqnarray*}
The difference of the two solutions, $\bw=\bu_h^{(1)}-\bu_h^{(2)}$,
satisfies $\bw\in V_h^0$ and
\begin{eqnarray}\label{Uniqueness:001}
a_s(\bw,\bv)=0, \qquad\forall \bv=\{\bv_0, \bv_b\}\in V_h^0.
\end{eqnarray}
By letting $\bv=\bw$ in (\ref{Uniqueness:001}) we obtain
$$
a_s(\bw,\bw) = 0.
$$
Using the definition of $a_s(\cdot,\cdot)$ we arrive at
\begin{eqnarray*}
\sum_{T\in\mathcal{T}_h}2(\mu\varepsilon_w(\bw),\varepsilon_w(\bw))_T+\sum_{T\in\mathcal{T}_h}(\lambda\nabla_w\cdot\bw,\nabla_w\cdot\bw)_T
\\
+\sum_{T\in\mathcal{T}_h}h_T^{-1}\langle
Q_b\bw_0-\bw_b,Q_b\bw_0-\bw_b\rangle_{\partial T}=0,
\end{eqnarray*}
which implies that
\begin{eqnarray}
\varepsilon_w(\bw)&=&0, \qquad \text{in}\  T,\label{EQ:July12:000}\\
Q_b\bw_0-\bw_b&=&0,  \qquad \text{on}\  \pT.\label{EQ:July12:001}
\end{eqnarray}
From the definition of the weak gradient, we have
\begin{eqnarray*}
(\nabla_w\bw,\tau)_T&=&(\nabla\bw_{0},\tau)_T-\langle\bw_{0}-\bw_{b},\tau\bn\rangle_\pT
\\
&=&(\nabla\bw_{0},\tau)_T-\langle
Q_b\bw_{0}-\bw_{b},\tau\bn\rangle_\pT
\end{eqnarray*}
for all $\tau\in [P_{k-1}(T)]^{d\times d}, k\geq1$. It follows from
(\ref{EQ:July12:001}) that $\nabla\bw_0=\nabla_w\bw$ on each element
$T$. Thus, with the help of (\ref{EQ:July12:000}),
$$
\varepsilon(\bw_0) = \varepsilon_w(\bw) = 0,
$$
which leads to $\bw_0\in {\sf RM}(T)\subset [P_1(T)]^d$. It follows
that $\bw_0|_{e}=Q_b\bw_0=\bw_b$, and hence $\bw_0$ is a continuous
function in $\Omega$ with vanishing boundary value on $\Gamma$. From
the second Korn's inequality (\ref{EQ:KornIneq:2nd-002}), we obtain
$\bw_0\equiv 0$ in $\Omega$, and hence $\bw_b \equiv 0$ from
(\ref{EQ:July12:001}). This shows that $\bu_h^{(1)} \equiv
\bu_h^{(2)}$, and hence the solution uniqueness and existence. We
remark that the result holds true for any $\lambda\ge 0$.
\end{proof}

\section{An Equivalent Mixed Formulation} A strong form of the mixed formulation
(\ref{var1})-(\ref{var2}) reads as follows: Find $\textbf{u}$ and
$p$ satisfying $\bu=\widehat\bu$ on $\Gamma$, the compatibility
condition $\int_\Omega \lambda^{-1} p dx = \int_\Gamma
\widehat\bu\cdot\bn ds$, and the following generalized Stokes
equations
\begin{equation}\label{model}
\begin{split}
-\nabla\cdot (2\mu \varepsilon(\bu))+\nabla p &=\bbf,\qquad \qquad \text{in}\ \Omega,\\
\nabla \cdot \bu & = \lambda^{-1} p,\qquad \, \text{in}\ \Omega.
\end{split}
\end{equation}
We assume that the generalized Stokes problem (\ref{model}) has the
$H^{1+s}(\Omega)\times H^s(\Omega)$-regularity for some $s\in
(\frac12, 1]$ in the sense that, for smooth data $\bbf$ and
$\widehat\bu$, the solution $\bu$ and $p$ of (\ref{model}) satisfy
$\bu\in [H^{1+s}(\Omega)]^d$, $p\in H^s(\Omega)$, and the following
a priori estimate
\begin{equation}\label{Regularity}
\|\bu\|_{1+s} + \|p\|_s \leq C (\|\bbf\|_{s-1} + \|\widehat
\bu\|_{s+\frac12,\Gamma})
\end{equation}
for some constant $C$ independent of the parameter $\lambda$. The
regularity estimate (\ref{Regularity}) can be found in
\cite{Grisvard, Dauge} for convex polygonal domain with $s=1$ when
$\lambda = \infty$. For large values of $\lambda$, one may
heuristically apply the result for the Stokes (by viewing
$\lambda^{-1} p$ as a given function) to obtain
\begin{equation*}
\|\bu\|_{1+s} + \|p\|_s \leq C (\|\bbf\|_{s-1} + \|\widehat
\bu\|_{s+\frac12,\Gamma} + \lambda^{-1} \|p\|_{s}),
\end{equation*}
which implies (\ref{Regularity}) when $\lambda$ is sufficiently
large.

The idea of weak Galerkin can be applied to the mixed formulation
(\ref{var1})-(\ref{var2}) for the linear elasticity problem
(\ref{primal_model})-(\ref{bc1}). This application requires an
additional finite element space that approximates the auxiliary
variable $p$. More precisely, we introduce
$$
W_h=\{q:  \ q|_T\in P_{k-1}(T), \ T\in\T_h\},\qquad W_h^0=W_h\cap
L_0^2(\Omega)
$$
and the following bilinear forms
\begin{eqnarray*}
a(\bw,\bv) & = & 2( \mu\varepsilon_w(\bw), \varepsilon_w(\bv))_h+s(\bw,\bv),\\
b(\bv,q) & = & (\nabla_{w}\cdot \bv, q)_h,\\
d(p,q) & = & \lambda^{-1} (p,q),
\end{eqnarray*}
where $\bw, \bv\in V_h$, $p, q\in W_h$, and
\begin{align*}
(\mu\varepsilon_w(\bw), \varepsilon_w(\bv))_h = &\sum_{T\in {\cal
T}_h}(
\mu\varepsilon_w(\bw), \varepsilon_w(\bv))_T,\\
 (\nabla_{w}\cdot \bv, q)_h= &\sum_{T\in {\cal T}_h}
 (\nabla_{w}\cdot \bv, q)_T.
\end{align*}

\begin{algorithm}\label{algo2}
For a numerical solution of the linear elasticity problem
(\ref{var1})-(\ref{var2}), find $\bu_h=\{\bu_0, \bu_b\}\in V_h$ and
$p_h\in W_h$ satisfying $\bu_b = Q_b \hat{\bu}$ on $ \Gamma$, the
compatibility condition $(\lambda^{-1}p_h,1) = \int_\Gamma
\widehat\bu\cdot\bn ds$, and the following equations
\begin{eqnarray}\label{3.3}
a(\bu_h,\bv)+b(\bv,p_h) &=& (\bbf,\bv_0), \qquad\forall \bv\in V_h^0,\\
 b(\bu_h,q)-d(p_h,q) &=& 0,\qquad\qquad
\forall q\in W_h^0.\label{3.4}
\end{eqnarray}
\end{algorithm}

\begin{lemma}
The weak Galerkin Algorithms \ref{algo-primal} and \ref{algo2} are
equivalent in the sense that the solution $\bu_h$ from
(\ref{WGA_primal}) and (\ref{3.3})-(\ref{3.4}) are identical to each
other.
\end{lemma}

\begin{proof}
Assume that $\tilde\bu_h$ and $\tilde p_h$ solves
(\ref{3.3})-(\ref{3.4}). Note that the equation (\ref{3.4}) can be
rewritten as
$$
(\nabla_{w}\cdot \tilde\bu_h, q)_T-\lambda^{-1}(\tilde p_h,q)_T=0,
\qquad\forall q\in P_{k-1}(T).
$$
Since $\nabla_{w}\cdot \tilde\bu_h\in P_{k-1}(T)$, then $\tilde p_h$
can be solved from the above equation as
\begin{equation}\label{EQ:July13:001}
\tilde p_h = \lambda \nabla_{w}\cdot \tilde\bu_h.
\end{equation}
Substituting (\ref{EQ:July13:001}) into (\ref{3.3}) yields
$$
a(\tilde\bu_h,\bv)+\lambda b(\bv,\nabla_{w}\cdot \tilde\bu_h) =
(\bbf,\bv_0), \qquad\forall \bv\in V_h^0.
$$
Note that
$$
a(\bw,\bv)+\lambda b(\bv,\nabla_{w}\cdot\bw) = a_s(\bw,\bv),\qquad
\forall \bw,\ \bv\in V_h.
$$
Thus, we have
$$
a_s(\tilde\bu_h, \bv) =(\bbf,\bv_0), \qquad\forall \bv\in V_h^0,
$$
which is the same as (\ref{WGA_primal}). It then follows from the
solution uniqueness that $\tilde\bu_h$ is identical with the
numerical solution arising from the weak Galerkin Algorithm
\ref{algo-primal}.

A similar argument can be applied to show that if $\bu_h$ solves
(\ref{WGA_primal}), then the pair $(\bu_h;
\lambda\nabla_w\cdot\bu_h)$ is a solution of
(\ref{3.3})-(\ref{3.4}). Details are left to interested readers as
an exercise.
\end{proof}

The reformulation (\ref{var1})-(\ref{var2}) of the elasticity
problem as a generalized Stokes system with nonzero divergence
constraint often leads to numerical approximations which are
locking-free as $\lambda\to\infty$. Due to the equivalence between
the primal formulation (\ref{WGA_primal}) and the mixed formulation
(\ref{3.3})-(\ref{3.4}) in the WG finite element method, the WG
scheme in the primal formulation (\ref{WGA_primal}) is locking-free
if the corresponding WG finite element method for
(\ref{var1})-(\ref{var2}) can be proved to be stable and accurate in
terms of the parameter $\lambda$. The rest of the paper is devoted
to a stability analysis for the mixed weak Galerkin finite element
scheme (\ref{3.3})-(\ref{3.4}), which in turn implies a locking-free
convergence for the linear elasticity problem in the displacement
formulation (\ref{WGA_primal}).

\section{Stability Conditions} In the weak finite element space $V_h$, we
introduce the following semi-norm
\begin{equation}\label{a6.1}
\3bar\bv\3bar =\Big(\sum_{T\in {\cal T}_h}\|\varepsilon(
\bv_0)\|_T^2+ h_T^{-1}\| Q_b\bv_0-\bv_b\|^2_{\partial
T}\Big)^{\frac{1}{2}},\quad \bv\in V_h.
\end{equation}

\begin{lemma}\label{lemmanorm}
The semi-norm  $\3bar\cdot\3bar$, as defined by (\ref{a6.1}), is
truly a norm in the linear space $V^0_h$.
\end{lemma}

\begin{proof}
We shall only verify the positivity property for $\3bar \cdot
\3bar$. To this end, assume that $\3bar\bv\3bar=0$ for some
 $\bv\in V^0_h$.  It follows that $\varepsilon(\bv_0)=0$ on each element $T\in
{\cal T}_h$ and $Q_b\bv_0=\bv_b$ on $\partial T$. Thus, $\bv_0\in
\textsf{RM}(T)$ satisfies $\bv_0|_e = Q_b(\bv_0|_e)=\bv_b$, which
implies the continuity of $\bv_0$ in the whole domain $\Omega$. The
boundary condition $\bv_b=\textbf{0}$ on $\Gamma$ implies that
$\bv_0=\textbf{0}$ on $\Gamma$. From the second Korn's inequality
(\ref{EQ:KornIneq:2nd-002}), we obtain $\bv_0 =\textbf{0}$ in
$\Omega$. The fact that $Q_b\bv_0=\bv_b$ on $\partial T$ gives
$\bv_b=\textbf{0}$ on $\partial T$. Thus, $\bv \equiv\textbf{0}$ in
$\Omega$, which completes the proof of the lemma.
\end{proof}

\begin{lemma}\label{lemmacoerv}
There exist positive constants $\alpha_1$ and $\alpha_2$ such that
\begin{equation}\label{EQ:coercivity}
\alpha_1 \3bar \bv \3bar^2 \leq a(\bv,\bv)\leq \alpha_2\3bar \bv
\3bar^2, \qquad \forall \ \bv\in V_h^0.
\end{equation}
\end{lemma}

\begin{proof} From (\ref{2.5}) and the integration by parts, we have
 \begin{equation}\label{EQ:July25.000}
 \begin{split}
  (\varepsilon_w  (\bv),\varphi)_T
=&(\nabla \bv_0,\frac{1}{2}(\varphi+\varphi^T))_T-\langle
\bv_0-\bv_b,
\frac{1}{2}(\varphi+\varphi^T)\bn\rangle_{\partial T}\\
=&(\varepsilon( \bv_0), \varphi )_T-\langle Q_b\bv_0-\bv_b,
\frac{1}{2}(\varphi+\varphi^T)\bn\rangle_{\partial T}
 \end{split}
 \end{equation}
for any $\varphi\in [P_{k-1}(T)]^{d\times d}$. Thus,
\begin{equation}\label{EQ:July25.001}
|(\varepsilon_w  (\bv),\varphi)_T| \leq \|\varepsilon
(\bv_0)\|_T\|\varphi\|_T+\| Q_b\bv_0-\bv_b\|_{\partial T}\|
\frac{1}{2}(\varphi+\varphi^T)\bn\|_{\partial T}.
\end{equation}
From the trace inequality (\ref{Aa}) and the usual inverse
inequality we have
\begin{eqnarray*}
\| \frac{1}{2}(\varphi+\varphi^T)\bn\|_{\partial T} &\leq & C \left(
h_T^{-1}\|\varphi\|_T^2 + h_T
\|\nabla\varphi\|_T^2\right)^{\frac12}\\
&\leq & C h_T^{-\frac12}\|\varphi\|_T.
\end{eqnarray*}
Substituting the above into (\ref{EQ:July25.001}) yields
\begin{equation*}
|(\varepsilon_w  (\bv),\varphi)_T| \leq
\left(\|\varepsilon(\bv_0)\|_T+ C
h_T^{-\frac12}\|Q_b\bv_0-\bv_b\|_{\partial T}\right)\|\varphi\|_T,
\end{equation*}
which leads to
\begin{equation}\label{EQ:July25.002}
\|\varepsilon_w(\bv)\|_{T}^2 \leq 2\|\varepsilon(\bv_0)\|_T^2+ C
h_T^{-1}\|Q_b\bv_0-\bv_b\|_{\partial T}^2.
\end{equation}

By representing $(\varepsilon( \bv_0), \varphi )_T$ in terms of
other two terms in (\ref{EQ:July25.000}), we can derive the
following analogy of (\ref{EQ:July25.002})
\begin{equation}\label{EQ:July25.003}
\|\varepsilon(\bv_0)\|_{T}^2 \leq 2\|\varepsilon_w(\bv)\|_T^2+ C
h_T^{-1}\|Q_b\bv_0-\bv_b\|_{\partial T}^2.
\end{equation}
The left inequality in (\ref{EQ:coercivity}) is a result of
(\ref{EQ:July25.003}) by summing over all the element $T\in\T_h$,
and the right inequality can be obtained by summing
(\ref{EQ:July25.002}) over $T\in\T_h$.
\end{proof}

In the finite element space  $W_h^0$, we introduce the following
norm
$$
\3bar q\3bar_0^2 = |q|_{0,h}^2+h^2\|\nabla q\|_{0,h}^2,\qquad q\in
W_h,
$$
where
$$
|q|_{0,h}^2= h\sum_{e\in {\cal E}_h^0} \|[\![q]\!]_e\|_e^2,\qquad
\|\nabla q\|^2_{0,h}=\sum_{T\in {\cal T}_h}\|\nabla q\|_T^2.
$$

\begin{lemma}\label{lemma-infsup}
There exists a constant $\beta>0$ such that
\begin{equation}\label{EQ:inf-sup}
\sup_{\bv\in V_h^0, \bv\neq 0} \frac{b(\bv,q)}{\3bar \bv \3bar} \geq
\beta \3bar q\3bar_0, \qquad \forall \ q\in W_h^0.
\end{equation}
\end{lemma}

\begin{proof}
From the definition of the discrete weak divergence (\ref{2.2}), we
have
\begin{equation}\label{EQ:July19.001}
b(\textbf{v},q)= \sum_{T\in {\cal T}_h} (\nabla_w \cdot \textbf{v},q)_T\\
= \sum_{T\in {\cal T}_h} \left\{-(\textbf{v}_0,\nabla q)_T+\langle
\textbf{v}_b \cdot \textbf{n}, q\rangle_{\partial T}\right\}.
\end{equation}
By setting $\textbf{v}=\textbf{v}_{q}:=\{-h^2\nabla q,
h[\![q]\!]_e\textbf{n}\}$ in (\ref{EQ:July19.001}), where $e\in
{\cal E}_h^0$ and $\textbf{n}$ is the unit outward normal direction
on $e$, we arrive at
\begin{equation}\label{EQ:July19:002}
b(\textbf{v}_{q},q)=h^2\sum_{T\in {\cal T}_h} \|\nabla q\|_T^2 +
h\sum_{e\in {\cal E}_h^0} \|[\![q]\!]_e\|_e^2=\3bar q\3bar_0^2.
\end{equation}
Furthermore, it is not hard to see that there exists a constant
$C_0$ such that
\begin{equation}\label{EQ:July19:003}
\3bar \bv_q\3bar \leq C_0 \3bar q \3bar_0.
\end{equation}
Thus, we have
\begin{eqnarray*}
\sup_{\bv\in V_h^0, \bv\neq 0} \frac{b(\bv,q)}{\3bar \bv \3bar} \geq
\frac{b(\bv_q,q)}{\3bar \bv_q \3bar}
 = \frac{\3bar q\3bar_0^2}{\3bar \bv_q \3bar}
 \ge  C_0^{-1} \3bar q\3bar_0,
\end{eqnarray*}
which proves the inf-sup condition (\ref{EQ:inf-sup}).
\end{proof}

\section{Preparation for Error Estimates}
For each element $T\in {\cal T}_h$, let $Q_0$ be the $L^2$
projection onto $[P_k(T)]^d$. For each edge/face $e\subset
\partial T$, recall that $Q_b$ is the $L^2$ projection onto the
finite element space on $e$. Denote by $Q_h\textbf{u}$ the $L^2$
projection onto the weak finite element space $V_h$ such that on
each element $T\in {\cal T}_h$,
$$
Q_h \textbf{u }:= \{Q_0\textbf{u}, Q_b \textbf{u}\}.
$$
Furthermore, let ${\cal Q}_h$ and $\textbf{Q}_h$ be the $L^2$
projection onto $P_{k-1}(T)$ and $[P_{k-1}(T )]^{d\times d}$,
respectively.

\begin{lemma}\cite{wy1302} \label{lemma4.2} For any $\bv \in [H^1(\Omega)]^d$, the
following identities hold true for the projection operators $Q_h$,
$\textbf{Q}_h$, and ${\cal Q}_h$

\begin{align}\label{4.4}
\nabla_w \cdot(Q_h \bv)=&{\cal Q}_h (\nabla \cdot
\bv), \\
\nabla_w (Q_h \bv)=& \textbf{Q}_h (\nabla \bv).\label{4.5}
\end{align}
Consequently, one has
\begin{equation}\label{EQ:July23.001}
\varepsilon_w(Q_h \bv) = \textbf{Q}_h \varepsilon(\bv).
\end{equation}
\end{lemma}

\begin{proof}  We outline a proof
for (\ref{4.4}) only; a similar approach can be adopted to prove
(\ref{4.5}). From (\ref{2.2}) and the integration by parts we have
\begin{align*}
(\nabla_w \cdot(Q_h \bv),
\varphi)_T&=-(Q_0\bv,\nabla \varphi)_T+\langle
Q_b\bv\cdot \bn,\varphi\rangle_{\partial T}\\
&=-( \bv,\nabla \varphi)_T+\langle
 \bv\cdot \bn,\varphi\rangle_{\partial T}\\
&=(\nabla\cdot \bv, \varphi)_T\\
&=({\cal Q}_h (\nabla\cdot \bv),\varphi)_T,
\end{align*}
for all $\varphi \in P_{k-1}(T)$. Since by construction $\nabla_w
\cdot(Q_h \bv)\in P_{k-1}(T)$, then (\ref{4.4}) follows.
\end{proof}

\medskip

\begin{lemma}\label{lemma5.1}
Assume that $(\bw; \rho)\in [H^{1+\gamma}(\Omega)]^d \times
H^1(\Omega)$, $\gamma>\frac12$, satisfies the following equation
\begin{equation}\label{5.2}
 2\nabla\cdot (\mu\varepsilon(\bw))+ \nabla \rho=-{\bm \eta}, \qquad
 \text{in}\ \Omega.
\end{equation}
Let $(Q_h\bw;{\cal Q}_h\rho)$ be the $L^2$ projection of $(\bw;
\rho)$ in the finite element space $V_h\times W_h$. Then, we have
\begin{equation}\label{erreqn}
2(\mu \varepsilon_w(Q_h\bw), \varepsilon_w(\bv))_h +(\nabla_w\cdot
\bv,{\cal Q}_h\rho)_h=({\bm\eta},\bv_0)+\ell_\textbf{\bw}
(\bv)+\theta_\rho(\bv),
\end{equation}
for all $\bv\in V_h^0$, where $\ell_\bw$ and $\theta_\rho$ are two
functionals in the linear space $V_h^0$ given by
\begin{align}\label{l}
\ell_\bw(\bv)&= 2\sum_{T\in{\cal T}_h}\langle
 \bv_0-\bv_b,\mu(\varepsilon(\bw)-\textbf{Q}_h\varepsilon(\bw)
 )\bn\rangle_{\partial T},\\
\theta_\rho(\bv)&=\sum_{T\in{\cal T}_h}\langle
 \bv_0-\bv_b  ,
(\rho- {\cal Q}_h\rho) \bn\rangle_{\partial T}.\label{theta}
\end{align}
\end{lemma}

\begin{proof}
From (\ref{EQ:July23.001}), (\ref{2.5}) and the integration by
parts, we get
\begin{equation}\label{5.4}
\begin{split}
& 2\mu(\varepsilon_w(Q_h\bw), \varepsilon_w(\bv))_T\\
=&2\mu ( \textbf{Q}_h \varepsilon(\bw),\varepsilon_w(\bv))_T\\
=&  -2\mu(\bv_0, \nabla\cdot (\textbf{Q}_h \varepsilon(\bw)
))_T+2\mu\langle \bv_b, \textbf{Q}_h \varepsilon(
\bw )\bn\rangle_{\partial T}\\
=&  2\mu( \nabla \bv_0, \textbf{Q}_h\varepsilon( \bw))_T-2\mu\langle
\bv_0 - \bv_b, \textbf{Q}_h \varepsilon(\bw)\bn\rangle_{\partial T}\\
=&  2\mu(\nabla\bv_0, \varepsilon(\bw)) _T-2\mu\langle \bv_0 - \bv_b
,\textbf{Q}_h \varepsilon( \bw)\bn\rangle_{\partial T}.
\end{split}
\end{equation}
Next, we have from (\ref{2.2}), the integration by parts, and
$
\sum_{T\in {\cal T}_h}\langle
\bv_b,\rho\bn\rangle_{\partial T}=0
$
that
\begin{equation*}
\begin{split}
(\nabla_w\cdot \bv,{\cal Q}_h \rho)_h=& \sum_{T\in {\cal T}_h}
(\nabla_w\cdot \bv,{\cal Q}_h \rho)_T \\
=&\sum_{T\in {\cal T}_h} \left\{-(\bv_0,\nabla({\cal Q}_h
\rho))_T+\langle \bv_b,({\cal Q}_h \rho)\bn\rangle_{\partial
T}\right\}\\
=& \sum_{T\in {\cal T}_h} \left\{(\nabla \cdot \bv_0, {\cal Q}_h
\rho )_T-\langle \bv_0-\bv_b,({\cal Q}_h \rho)\bn\rangle_{\partial
T}\right\}\\
=& \sum_{T\in {\cal T}_h} \left\{(\nabla \cdot \bv_0, \rho )_T-
\langle \bv_0-\bv_b,({\cal Q}_h \rho)\bn\rangle_{\partial
T}\right\}\\
=& \sum_{T\in {\cal T}_h}\left\{-(\bv_0,\nabla  \rho )_T+\langle
\bv_0,  \rho \bn \rangle_{\partial T} -\langle \bv_0-\bv_b,({\cal
Q}_h \rho)\bn\rangle_{\partial
T}\right\}\\
=& \sum_{T\in {\cal T}_h}\left\{-(\bv_0,\nabla  \rho )_T+\langle
\bv_0-\bv_b,  \rho \bn \rangle_{\partial T}-\langle
\bv_0-\bv_b,({\cal Q}_h
\rho)\bn\rangle_{\partial T}\right\}\\
=&  -( \bv_0,\nabla  \rho ) + \sum_{T\in {\cal T}_h} \langle
\bv_0-\bv_b, (\rho-{\cal Q}_h \rho) \bn \rangle_{\partial T},
\end{split}
\end{equation*}
which leads to
\begin{equation}\label{5.5}
( \bv_0,\nabla  \rho ) =-(\nabla_w\cdot \bv,{\cal Q}_h \rho)_h+
\sum_{T\in {\cal T}_h} \langle \bv_0-\bv_b, (\rho-{\cal Q}_h \rho)
\bn \rangle_{\partial T}.
\end{equation}

Now testing (\ref{5.2}) by using the component $\textbf{v}_0$ of
$\textbf{v} =\{\textbf{v}_0, \textbf{v}_b\} \in  V^ 0_ h$  yields
\begin{equation}\label{5.6}
-2( \nabla\cdot( \mu\varepsilon( \textbf{w})),\textbf{v}_0)-( \nabla
\rho,\textbf{v}_0) =( {\bm \eta},\textbf{v}_0) .
\end{equation}
From the integration by parts, we can rewrite (\ref{5.6}) as
\begin{equation}\label{above}
\begin{split}
 2\sum_{T\in{\cal T}_h}( \mu \varepsilon( \textbf{w}) ,
\nabla\textbf{v}_0)_T-2 \sum_{T\in{\cal T}_h}\langle\mu \varepsilon(
\textbf{w}) \textbf{n}, \bv_0\rangle_{\partial T}-( \nabla
\rho,\textbf{v}_0)=( {\bm \eta},\textbf{v}_0) .
\end{split}
\end{equation}
Substituting (\ref{5.4}) and (\ref{5.5}) into (\ref{above}) yields
\begin{equation*}\label{5.7}
\begin{split}
 &2\mu \sum_{T\in{\cal T}_h} \left\{(\varepsilon_w(Q_h\bw), \varepsilon_w(\bv))_T+
 \langle \bv_0 - \bv_b,\textbf{Q}_h \varepsilon( \bw)\bn\rangle_{\partial T}-
 \langle\bv_0,\varepsilon(\bw) \bn \rangle_{\partial T}\right\}\\
&+(\nabla_w\cdot \textbf{v},{\cal Q}_h \rho)_h- \sum_{T\in {\cal
T}_h} \langle \textbf{v}_0-\textbf{v}_b, (\rho-{\cal Q}_h \rho)
\textbf{n} \rangle_{\partial T}=( {\bm \eta},\textbf{v}_0),
\end{split}
\end{equation*}
which implies that
\begin{equation*}
\begin{split}
 &2\mu \sum_{T\in{\cal T}_h} \left\{(\varepsilon_w(Q_h\bw), \varepsilon_w(\bv))_T-
 \langle \bv_0 - \bv_b ,(\varepsilon(\bw)- \textbf{Q}_h
\varepsilon( \bw))\bn\rangle_{\partial
T}-\langle\bv_b,\varepsilon(\bw)
\bn \rangle_{\partial T}\right\}\\
&+(\nabla_w\cdot \textbf{v},{\cal Q}_h \rho)_h- \sum_{T\in {\cal
T}_h} \langle \textbf{v}_0-\textbf{v}_b, (\rho-{\cal Q}_h \rho)
\textbf{n} \rangle_{\partial T} =( {\bm \eta},\textbf{v}_0).
\end{split}
\end{equation*}
Using the boundary condition $\bv_b=0$ we obtain
\begin{equation}\label{equa}
\begin{split}
&\ 2(\mu\varepsilon_w(Q_h\bw), \varepsilon_w(\bv))_h+(\nabla_w\cdot
\textbf{v},{\cal Q}_h \rho)_h \\
= &\sum_{T\in {\cal T}_h}\left\{2\langle \bv_0 - \bv_b,
\mu(\varepsilon( \bw)- \textbf{Q}_h \varepsilon(
\bw))\bn\rangle_{\partial T}+ \langle \bv_0-\bv_b,
(\rho-{\cal Q}_h \rho) \bn \rangle_{\partial T}\right\}\\
&+( {\bm\eta},\textbf{v}_0),
\end{split}
\end{equation}
which is precisely the equation (\ref{erreqn}). This completes the
proof.
\end{proof}

\section{Error Estimate in a Discrete $H^1$-Norm}
For the weak Galerkin finite element solution $(\bu_h;p_h)=(
\{\bu_0, \bu_b\};p_h) \in V_h\times W_h$ arising from
(\ref{3.3})-(\ref{3.4}) for the linear elasticity problem
(\ref{primal_model})-(\ref{bc1}), we define its error functions
$\textbf{e}_h$ and $\zeta_h$ by
\begin{align}\label{5.1}
\be_h&=\{\be_0,\be_b\}=\{Q_0\bu-\bu_0,
Q_b\bu-\bu_b\},\\
 \zeta_h&={\cal Q}_h p-p_h,\label{ph}
\end{align}
where $(\bu;p)$ is the exact solution of the variational problem
(\ref{var1})-(\ref{var2}). It is clear that $\be_h\in V_h^0$ and
$\zeta_h\in W_h^0$.

\begin{lemma} \label{Lem5.2}  The error functions
$\be_h$ and $\zeta_h$ defined in (\ref{5.1})-(\ref{ph}) satisfy the
following error equations
\begin{align}\label{4.11}
a(\be_h,\bv)+b(\bv, \zeta_h)&=
\varphi_{\bu,p}(\bv), \qquad \forall \bv\in V_h^0,\\
b(\be_h,q)-d(\zeta_h,q)&=0,\qquad \qquad\quad \forall q\in
W_h^0,\label{4.12}
\end{align}
 where
$$
\varphi_{\bu,p}(\bv)=\ell_\bu(\bv)+
\theta_p(\bv)+ s(Q_h\bu,\bv).
$$
\end{lemma}

\begin{proof} Observe that the exact solution $(\bu;p)$ satisfies the equation
(\ref{5.2}) with ${\bm\eta} = \textbf{f}$. Thus, from Lemma
\ref{lemma5.1} we have
$$
2(\mu\varepsilon_w(Q_h\bu), \varepsilon_w(\bv))_h +(\nabla_w\cdot
\bv,{\cal Q}_h p)_h=(\textbf{f},\bv_0)+\ell_\bu (\bv)+\theta_p(\bv),
$$
which leads to
\begin{equation}\label{4.13}
a(Q_h\textbf{u},\textbf{v})+b(\textbf{v},{\cal
Q}_hp)=(\textbf{f},\textbf{v}_0)+\ell_\textbf{u}(\textbf{v})+
\theta_p(\textbf{v})+s(Q_h\textbf{u},\textbf{v}).
\end{equation}
Subtracting (\ref{3.3}) from (\ref{4.13}) gives the equation
(\ref{4.11}).

To derive (\ref{4.12}), using (\ref{4.4}) we have for any $q\in W_h$
\begin{equation}\label{4.14}
\begin{split}
(\nabla_w\cdot(Q_h\bu),q) - \lambda^{-1}({\cal Q}_h p,q) & = ({\cal
Q}_h(\nabla\cdot \bu),q)- \lambda^{-1}( {\cal Q}_h p,q)\\
&=(\nabla\cdot \bu,q)- \lambda^{-1}(p,q) =0,
\end{split}
\end{equation}
where we have used (\ref{var2}) in the second line. The difference
of (\ref{4.14}) and (\ref{3.4}) yields (\ref{4.12}).
\end{proof}

We are now in a position to derive an error estimate for the weak
Galerkin finite element approximation $(\textbf{u}_h;p_h)$.

\begin{theorem}
Let the solution of (\ref{var1})-(\ref{var2})
 be sufficiently smooth such that $(\bu; p) \in
[H^{k+1}(\Omega)]^d \times H^{k}(\Omega)$ for some $k\ge 1$. For the
weak Galerkin finite element solution $(\bu_h ; p_h ) \in V_h \times
W_h$ arising from (\ref{3.3})-(\ref{3.4}), we have
\begin{equation}\label{th1}
\3barQ_h\bu-\bu_h\3bar +\lambda^{- \frac{1}{2}}\|{\cal
Q}_hp-p_h\|+\3bar {\cal Q}_hp-p_h \3bar_0\leq
Ch^k(\|\bu\|_{k+1}+\|p\|_k),
\end{equation}
where $C$ is a generic constant independent of $(\bu; p)$.
Consequently, the following error estimate holds true
\begin{equation}\label{th1-newversion}
\3bar \bu-\bu_h\3bar +\lambda^{- \frac{1}{2}}\|p-p_h\|+\3bar p-p_h
\3bar_0\leq Ch^k(\|\bu\|_{k+1}+\|p\|_k).
\end{equation}
\end{theorem}

\begin{proof} By choosing $\bv=\be_h$ in
(\ref{4.11}) and $q = \zeta_h$ in (\ref{4.12}) we have
$$
a(\textbf{e}_h,\textbf{e}_h)+ {\lambda}^{-1}\|
\zeta_h\|^2=\varphi_{\textbf{u},p}(\textbf{e}_h).
$$
By applying Lemma \ref{lem8.2} to the term
$\varphi_{\textbf{u},p}(\textbf{e}_h)$ we arrive at
\begin{equation*}
a(\textbf{e}_h,\textbf{e}_h) + {\lambda}^{-1}\|\zeta_h\|^2  \leq
Ch^k(\|\textbf{u}\|_{k+1}+\|p\|_k)\3bar \textbf{e}_h\3bar.
\end{equation*}
Next, using Lemma \ref{lemmacoerv} and the above estimate we obtain
\begin{eqnarray}\label{eh}
\alpha_1\3bar\textbf{e}_h\3bar^2+ {\lambda}^{-1}\|\zeta_h\|^2 \leq
Ch^k(\|\textbf{u}\|_{k+1}+\|p\|_k)\3bar \textbf{e}_h\3bar.
\end{eqnarray}

To derive an error estimate for the ``pressure'' variable $p$ in a
$\lambda$-independent norm, we use the {\em inf-sup} condition
(\ref{EQ:inf-sup}) to obtain
\begin{equation}\label{EQ:July24:001}
\beta \3bar \zeta_h \3bar_0 \leq \sup_{\bv\in V_h^0, \bv\neq 0}
\frac{b(\bv,\zeta_h)}{\3bar \bv \3bar}.
\end{equation}
From (\ref{4.11})
 \begin{equation*}
b(\bv,\zeta_h)=-a(\be_h,\bv)+ \varphi_{\bu,p}(\bv).
\end{equation*}
Thus, it follows from Lemma \ref{lemmacoerv}, the error estimate
(\ref{eh}), and Lemma \ref{lem8.2} that
\begin{eqnarray*}
|b(\bv,\zeta_h)| &\leq & \alpha_2 \3bar \be_h \3bar \ \3bar
\bv\3bar + |\varphi_{\bu,p}(\bv)| \\
&\leq & Ch^k(\|\bu\|_{k+1}+\|p\|_k)\ \3bar\bv\3bar.
\end{eqnarray*}
Substituting the above estimate into (\ref{EQ:July24:001}) yields
\begin{equation}\label{EQ:July24:002}
\3bar \zeta_h \3bar_0 \leq Ch^k(\|\bu\|_{k+1}+\|p\|_k).
\end{equation}
Combining (\ref{eh}) with (\ref{EQ:July24:002}) gives rise to the
error estimate (\ref{th1}). Finally, (\ref{th1-newversion}) stems
from the usual triangle inequality, the estimate (\ref{th1}) and the
error estimate for $L^2$ projections.
\end{proof}

\section{ Error Estimate in $L^2$} As usual, we use the duality argument
to derive an $L^2$ error estimate for the weak Galerkin finite
element method. To this end, consider the problem of seeking
$\bPhi\in [H^1(\Omega)]^d$ and $\xi\in L^2_0(\Omega)$ satisfying
\begin{eqnarray}\label{7.1}
-\nabla\cdot (2\mu \varepsilon(\bPhi)+\xi I)&=&\textbf{e}_0,\qquad \text{in} \ \Omega,\\
\nabla\cdot \bPhi - \lambda^{-1} \xi &=& 0, \ \qquad\text{in} \ \Omega,\label{psi}\\
\bPhi&=&\textbf{0}, \ \qquad\text{on} \ \Gamma.\label{bdy}
\end{eqnarray}
Assume that the dual problem (\ref{7.1})-(\ref{bdy}) has the
$[H^{1+s} (\Omega)]^d \times  H^s (\Omega)$-regularity with $\frac12
<s\leq 1$ in the sense that the solution $(\bPhi; \xi)\in [H^{1+s}
(\Omega)]^d \times  H^s (\Omega)$ and satisfies the following a
priori estimate:
\begin{equation}\label{7.2}
\|\bPhi\|_{s+1}+ \|\xi\|_{s}\leq C\|\textbf{e}_0 \|.
\end{equation}

\begin{theorem} Assume that the solutions of (\ref{model}) are sufficiently
smooth such that $(\bu;p) \in [H ^{k+1} (\Omega)]^d\times
H^k(\Omega)$ for some integer $k\geq 1$. Let $(\bu_h; p_h) \in V_h
\times W_h$ be the corresponding weak Galerkin finite element
solution arising from (\ref{3.3})-(\ref{3.4}). Then, under the
regularity assumption (\ref{7.2}), there exists a constant $C$, such
that
\begin{equation}\label{7.3}
\|Q_0\bu-\bu_0\|\leq Ch^{k+s}\big(\|\bu\|_{k+1}+\|p\|_k\big).
\end{equation}
Moreover, it follows from the triangle inequality and the error
estimate for the $L^2$ projection that
\begin{equation}\label{7.3-new}
\|\bu-\bu_0\|\leq Ch^{k+s}\big(\|\bu\|_{k+1}+\|p\|_k\big).
\end{equation}
\end{theorem}

\begin{proof} Note that the solution $(\bPhi;\xi)$ of
(\ref{7.1})-(\ref{bdy}) satisfies (\ref{5.2}) with ${\bm\eta} =
\be_0$. Thus, using Lemma \ref{lemma5.1}, namely the identity
(\ref{erreqn}), we have
\begin{equation}\label{EQ:July25:100}
2(\mu \varepsilon_w (Q_h \bPhi),\varepsilon_w (\bv))_h +(\nabla_w
\cdot \bv,{\cal Q}_h \xi)_h = (\be_0, \bv_0) +\ell_{\bPhi}(\bv) +
\theta_\xi(\bv)
\end{equation}
for all $\bv\in V_h^0$.

By choosing $\bv=\be_h$ in (\ref{EQ:July25:100}) we obtain
\begin{equation}\label{7.4}
\|\be_0\|^2=a(Q_h\bPhi, \textbf{e}_h)+b( \textbf{e}_h,{\cal Q}_h
\xi)-\varphi_{\bPhi,\xi}(\textbf{e}_h),
\end{equation}
where
$$
\varphi_{\bPhi,\theta}(\be_h)=\theta_\xi(\be_h)+\ell_{\bPhi}(\be_h)
+s(Q_h\bPhi, \be_h).
$$
Using (\ref{4.12}) we have
\begin{equation}\label{EQ:August02:001}
b(\textbf{e}_h,{\cal Q}_h\xi)-d({\cal Q}_h\xi,\zeta_h)=0.
\end{equation}
From (\ref{4.4}) and (\ref{psi})
\begin{equation*}
\begin{split}
b(Q_h\bPhi,\zeta_h)&=\sum_{T\in {\cal T}_h}(\nabla_w \cdot Q_h\bPhi,\zeta_h)_T\\
&= \sum_{T\in {\cal T}_h} ( {\cal Q}_h(\nabla \cdot\bPhi),\zeta_h)_T \\
&=d({\cal Q}_h\xi,\zeta_h).
\end{split}
\end{equation*}
Thus, by substituting the above into (\ref{EQ:August02:001})
$$
b(\textbf{e}_h,{\cal Q}_h\xi) = b(Q_h\bPhi,\zeta_h).
$$
Combining the last equation with (\ref{7.4}) we obtain
\begin{equation*}
\|\be_0\|^2=a(Q_h\bPhi, \be_h)+ b(Q_h\bPhi,\zeta_h)
-\varphi_{\bPhi,\xi}(\be_h),
\end{equation*}
which, together with the error equation (\ref{4.11}), leads to
\begin{equation}\label{7.7}
\|\be_0\|^2= \varphi_{\bu,p}(Q_h\bPhi)-\varphi_{\bPhi,\xi}(\be_h).
\end{equation}

The rest of the proof shall deal with the two terms on the
right-hand side of (\ref{7.7}). The second term
$\varphi_{\bar{\textbf{u}},q}(\textbf{e}_h)$ can be estimated by
using Lemma \ref{lem8.2}, the error estimate (\ref{th1}), and the
regularity assumption (\ref{7.2}) as follows
\begin{equation}\label{7.8}
\begin{split}
|\varphi_{\bPhi,\xi}(\be_h)|&\leq Ch^s(\|\bPhi\|_{s+1}+\|\xi\|_s)\
\3bar \be_h\3bar\\
&\leq Ch^{k+s}(\|\bu\|_{k+1}+\|p\|_k)\ \|\be_0\|.
\end{split}
\end{equation}
The estimate for the first term $\varphi_{\bu,p}(Q_h\bPhi)$ is a bit
complicated, for which we detail as follows.

(i) For the term $s(Q_h\bu, Q_h\bPhi)$ in
$\varphi_{\bu,p}(Q_h\bPhi)$, from the Cauchy-Schwarz inequality,
(\ref{A4}) and (\ref{A1}), we have
\begin{equation}\label{lerr1.000}
\begin{split}
|s(Q_h\bu, Q_h\bPhi)|=&\Big|\sum_{T\in {\cal T}_h}h_T^{-1}\langle
Q_b(Q_0\bu)-Q_b\bu,Q_b(Q_0\bPhi)-Q_b\bPhi\rangle_{\partial
T}\Big|\\
\leq & \Big(\sum_{T\in {\cal T}_h}
h_T^{-1}\|Q_0\bu-\bu\|^2_{\partial
T}\Big)^{\frac{1}{2}}\Big(\sum_{T\in {\cal T}_h}
h_T^{-1}\|Q_0\bPhi-\bPhi\|^2_{\partial T}\Big)^{\frac{1}{2}}\\
\leq & Ch^{k+s}\|\bu\|_{k+1} \ \|\bPhi\|_{s+1}.
\end{split}
\end{equation}

(ii) For the term $\ell_\bu(Q_h\bPhi)$, we have from the
orthogonality of $Q_b$ and the boundary condition (\ref{bdy}) that
\begin{equation}\label{lerr1}
\begin{split}
&2\mu\sum_{T\in {\cal T}_h} \langle Q_b\bPhi-\bPhi,(\varepsilon(
\bu)-\textbf{Q}_h(\varepsilon(
\bu)))\bn\rangle_{\partial T}\\
=&2\mu\sum_{T\in {\cal T}_h} \langle Q_b\bPhi-\bPhi,\varepsilon(
\bu)\bn\rangle_{\partial T} =0.
\end{split}
\end{equation}
Thus, it follows from the Cauchy-Schwarz inequality, (\ref{A4}),
(\ref{A1}) and (\ref{A2}) that
\begin{equation}\label{lerr2}
\begin{split}
|\ell_\bu(Q_h\bPhi)| & =\left|2\mu \sum_{T\in {\cal T}_h} \langle
Q_0\bPhi- Q_b\bPhi,(\varepsilon( \bu)-\textbf{Q}_h
\varepsilon(\bu))\bn\rangle_{\partial
T}\right|\\
& =\left|2\mu \sum_{T\in {\cal T}_h} \langle Q_0\bPhi-
\bPhi,(\varepsilon( \bu)-\textbf{Q}_h
\varepsilon(\bu))\bn\rangle_{\partial
T}\right|\\
& \leq  2\mu\Big(\sum_{T\in {\cal T}_h} h_T \|\varepsilon(
\bu)-\textbf{Q}_h\varepsilon( \bu)\|^2_{\partial
T}\Big)^{\frac{1}{2}}\Big(\sum_{T\in {\cal T}_h} h_T^{-1}\|Q_0\bPhi-
\bPhi\|^2_{\partial
T}\Big)^{\frac{1}{2}}\\
& \leq Ch^{k+s}\|\bu\|_{k+1} \ \|\bPhi\|_{s+1}.
\end{split}
\end{equation}

(iii) As to the term  $\theta_p (Q_h \bPhi)$, we again use the
orthogonality of $Q_b$ and the boundary condition (\ref{bdy})
combined with the Cauchy-Schwarz inequality, (\ref{A4}), (\ref{A1})
and (\ref{A3}) to obtain
\begin{equation}\label{theerr1}
\begin{split}
|\theta_p (Q_h \bPhi)|& =\left|\sum_{T\in {\cal T}_h}
  \langle Q_0\bPhi- Q_b\bPhi,(p-{\cal Q}_hp)\bn\rangle_{\partial
  T}\right|\\
  &=\left|\sum_{T\in {\cal T}_h}
  \langle Q_0\bPhi- \bPhi,(p-{\cal Q}_hp)\bn\rangle_{\partial T}\right|\\
&\leq \Big(\sum_{T\in {\cal T}_h} h_T \|p-{\cal Q}_hp\|^2_{\partial
T}\Big)^{\frac{1}{2}}\Big(\sum_{T\in {\cal T}_h} h_T^{-1}\|Q_0\bPhi-
\bPhi\|^2_{\partial
T}\Big)^{\frac{1}{2}}\\
 &\leq  Ch^{k+s}\|p\|_{k} \ \|\bPhi\|_{s+1}.
\end{split}
\end{equation}

By combining the three estimates (\ref{lerr1.000}), (\ref{lerr2}),
and (\ref{theerr1}) we arrive at
\begin{equation}\label{7.9}
\begin{split}
|\varphi_{\bu,p}(Q_h\bPhi)|&\leq
Ch^{k+s}(\|\bu\|_{k+1}+\|p\|_k)\|\bPhi\|_{s+1}\\
&\leq Ch^{k+s}(\|\bu\|_{k+1}+\|p\|_k)\|\be_0\|,
\end{split}
\end{equation}
where we have used the regularity assumption (\ref{7.2}) in the
second line. Finally, by substituting (\ref{7.8}) and (\ref{7.9})
into (\ref{7.7}) we obtain the desired error estimate of
(\ref{7.3}). This completes the derivation of the $L^2$ error
estimate.
\end{proof}

\section{Supporting Tools and Inequalities} In this section, we
present some technical inequalities that support the error analysis
established in previous sections.

Recall that ${\cal T}_h$ is a shape-regular finite element partition
of $\Omega$. There exists a constant $C>0$ such that, for any side
$e\subset T$ with $T\in {\cal T}_h$, the following trace inequality
holds true \cite{wy1202}
\begin{equation}\label{A4}
\|\phi\|_e^2\leq C\big(h_T^{-1}\|\phi\|_T^2+h_T\|\nabla
\phi\|_T^2\big), \qquad\forall \phi\in H^1(T),
\end{equation}
where $h_T$ is the size of $T$. Furthermore, in the polynomial space
$P_j(T)$, $j\ge 0$, we have from the inverse inequality that
\begin{equation}\label{Aa}
\|\phi\|_e^2\leq C h_T^{-1}\|\phi\|_T^2,\qquad \forall \phi\in
P_j(T), \ T\in \T_h.
\end{equation}

\begin{lemma}\label{lemmaA1} \cite{wy1202}  Assume
that  ${\cal T}_h$ is a finite element partition of $\Omega$
satisfying the shape regularity assumption as defined in
\cite{wy1202}. Let $k\ge 1$ be the order of the finite element
method and $1 \leq r \leq k$. Let $\bw\in [H^{r+1} (\Omega)]^ d$,
$\rho \in H^r (\Omega)$ and $0 \leq m \leq 1$. There holds
\begin{align}\label{A1}
\sum_{T\in{\cal
T}_h}h_T^{2m}\|\bw-Q_0\bw\|^2_{T,m}&\leq  C
h^{2(r+1)}\|\bw\|^2_{r+1},\\
\sum_{T\in{\cal T}_h}h_T^{2m}\|
\varepsilon(\bw)-\textbf{Q}_h\varepsilon(\bw)\|^2_{T,m}&\leq
Ch^{2r}\|\bw\|^2_{r+1},\label{A2}\\
\sum_{T\in{\cal T}_h}h_T^{2m}\|\rho-{\cal Q}_h\rho\|^2_{T,m}&\leq
Ch^{2r}\|\rho\|^2_{r}.\label{A3}
\end{align}
\end{lemma}

\subsection{Korn's inequality}
Korn's inequality is a fundamental tool in the study of elasticity
equations. The inequality can be found in many existing literature,
see \cite{Ciarlet, Duvaut, Nitsche, b2003} for example. For
convenience, we provide a summary here for this useful inequality.

\begin{theorem}(Korn's Inequality) Assume that the domain $\Omega$
is open bounded with Lipschitz continuous boundary. Then, there
exists a constant $C$ such that
\begin{equation}\label{EQ:KornsInequality}
\|\nabla\bv\|_{0} \leq C \left (\|\varepsilon(\bv)\|_{0}+\|\bv\|_{0}
\right),\qquad \forall \ \bv\in [H^1(\Omega)]^d.
\end{equation}
\end{theorem}

\begin{proof} A proof can be given by using the following inequality
\cite{Duvaut}
\begin{equation}\label{Lions-Estimate}
\|\phi\|_0 \leq C (\|\nabla \phi\|_{-1} + \|\phi\|_{-1}),\qquad
\forall \ \phi\in L^2(\Omega).
\end{equation}
To this end, for any $\bv\in [H^1(\Omega)]^d$, it is not hard to
check that
\begin{equation}\label{Lions-Estimate-2}
\partial_j\partial_k v_i = \partial_j \varepsilon_{ik}(\bv)
+\partial_k \varepsilon_{ij}(\bv) - \partial_i
\varepsilon_{jk}(\bv).
\end{equation}
It follows that $\nabla v_i\in H^{-1}(\Omega)$ if
$\varepsilon(\bv)\in [L^2(\Omega)]^{d\times d}$. Moreover, from
(\ref{Lions-Estimate}) and (\ref{Lions-Estimate-2}) one has
\begin{equation*}
\begin{split}
\|\nabla \bv\|_0 &\leq C (\|\nabla\varepsilon(\bv)\|_{-1} +
\|\nabla\bv\|_{-1})\\
&\leq C (\|\varepsilon(\bv)\|_0  + \|\bv\|_{0}),
\end{split}
\end{equation*}
which is the Korn's inequality (\ref{EQ:KornsInequality}).
\end{proof}

The following is a characterization of the space of rigid motions as
the kernel of the strain tensor operator.

\begin{lemma}\label{lemma-rigidmotion}
Let $\Omega$ be an open bounded and connected domain in
$\mathbb{R}^d$. The space of rigid motions ${\sf RM}(\Omega)$ is
identical to the kernel of the strain tensor operator; i.e., for
$\bv\in [H^1(\Omega)]^d$, there holds $\varepsilon(\bv)=0$ if and
only if $\bv\in {\sf RM}(\Omega)$.
\end{lemma}

\begin{proof}
For any $\bv \in {\sf RM}(\Omega)$, there exist ${\bf{a}}\in \mathbb
R^d$ and a skew-symmetric $d\times d$ matrix $\eta$ such that
$\bu=\bf{a}+\eta {\bf{x}}$. It is easy to check that $\e(\bv)=0$.

For any $\bv$ satisfying $\e(\bv)=0$, we have from
(\ref{Lions-Estimate-2}) that $\partial_j\partial_k v_i=0$, and
hence $\bv \in [P_1(\Omega)]^d$. Thus, there exist $\bf{a}\in
\mathbb R^d$ and $\eta\in \mathbb R^{d\times d}$ such that
$\bv={\bf{a}}+\eta{\bf{x}}$. Since $\e(\bv)=0$, we get
\begin{eqnarray*}
0=\e(\bv)=\e(\bf{a}+\eta {\bf{x}}) =\frac12(\nabla(\bf{a}+\eta
{\bf{x}})+\nabla(\bf{a}+\eta {\bf{x}})^T)=\frac12(\eta+\eta^T).
\end{eqnarray*}
It follows that $\eta^T=-\eta$, which means that $\eta$ is
skew-symmetric, and hence $\bv\in {\sf RM}(\Omega)$.
\end{proof}

\begin{lemma}\label{lemma:rigidmotionzero}
Let ${\bf{x}}_i, ~i=0,\ldots,d-1$, be $d$-points in $\Omega$ which
form a $(d-1)-$dimensional hyperplane. If $\bv\in {\sf{RM}}(\Omega)$
satisfies $\bv({\bf{x}}_i)={\bf 0},~i=0,\ldots,d-1$, then we must
have $\bv\equiv{\bf 0}$.
\end{lemma}

\begin{proof}
Let $\bv={\bf{a}}+\eta {\bf{x}}$ be the rigid motion with
$\bv({\bf{x}}_i)={\bf 0}$. Thus,
\begin{equation}\label{EQ:July27:200}
\eta ({\bf{x}}_i - {\bf{x}}_0) ={\bf 0}, \qquad i=1,\ldots, d-1.
\end{equation}
Note that the set of vectors $\{{\bf{x}}_i -
{\bf{x}}_0\}_{i=1}^{d-1}$ is linearly independent since they form a
basis of a $(d-1)-$dimensional subspace of $\mathbb{R}^d$.

For $d=2$, the skew-symmetric matrix $\eta$ is either zero or
invertible. From the equation (\ref{EQ:July27:200}), we see that
$\eta$ can be nothing except $\eta=0$. For $d=3$, the matrix $\eta$
has eigenvalue $\lambda_0=0$. If $\eta\neq 0$, then the eigen-space
corresponding to the eigenvalue $\lambda_0=0$ must have dimension
$1$. But the equation (\ref{EQ:July27:200}) indicates that the
dimension for this eigen-space is no less than $2$. Consequently, we
must have $\eta =0$. Finally, from $\bv({\bf{x}}_0)={\bf 0}$ we have
${\bf{a}}=0$. This shows that $\bv\equiv 0$.
\end{proof}

\begin{theorem}\label{kornsecond}(Second Korn's Inequality) Assume that
the domain $\Omega$ is connected, open bounded with Lipschitz
continuous boundary. Let $\Phi: H^1(\Omega)\to \mathbb{R}^+$ be a
semi-norm on $H^1(\Omega)$ satisfying
$$
\Phi(\bv)=0 \ and\ \bv\in {\sf RM}(\Omega) \Rightarrow \bv=0.
$$
Then, there exists a constant $C$ such that
\begin{equation}\label{EQ:KornIneq:2nd}
\|\bv\|_{1} \leq C (\|\varepsilon(\bv)\|_{0} + \Phi(\bv)),
\end{equation}
for all $\bv\in [H^1(\Omega)]^d$.
\end{theorem}

\begin{proof}
We verify the inequality (\ref{EQ:KornIneq:2nd}) by a contradiction
argument. To this end, assume that (\ref{EQ:KornIneq:2nd}) does not
hold true. Then, for each integer $n$, there exists $\bv_n\in
[H^1(\Omega)]^d$ such that
\begin{equation}\label{EQ:09-02-001}
\|\bv_n\|_{1} > n (\|\varepsilon(\bv_n)\|_{0} + \Phi(\bv_n)).
\end{equation}
We may assume $\|\bv_n\|_{1}=1$, and hence, there exists a
subsequence $\{\bv_{n_k}\}$ which is weakly convergent in $H^1$ and
strongly convergent in $L^2(\Omega)$. From (\ref{EQ:09-02-001}), we
have
$$
\|\varepsilon(\bv_n)\|_{0} +\Phi(\bv_n) < n^{-1}.
$$
Thus,
$$
\|\varepsilon(\bv_n)\|_{0} + \Phi(\bv_n) \to 0,
$$
which, together with Korn's inequality (\ref{EQ:KornsInequality}),
implies that $\{\bv_{n_k}\}$ is a Cauchy sequence in $H^1(\Omega)$.
Hence, there exists a function $\bv\in H^1(\Omega)$ such that
$$
\bv_{n_k} \rightarrow \bv \qquad \mbox{strongly \ in \
$H^1(\Omega)$}.
$$
Moreover, we have
$$
\|\varepsilon(\bv)\|_0 +\Phi(\bv)= \lim_{k\to\infty}
(\|\varepsilon(\bv_{n_k})\|_0  +\Phi(\bv_{n_k})) =0.
$$
Thus,
$$
\bv\in {\sf RM}(\Omega) \ \mbox{and} \ \Phi(\bv)=0,
$$
which leads to $\bv=0$. This is  a contradiction to the assumption
that
$$
\|\bv\|_1 =\lim_{k\to\infty} \|\bv_{n_k}\|_1=1.
$$
This completes the proof.
\end{proof}

\medskip

The following are two particular cases of the seminorm $\Phi(\cdot)$
that satisfy the conditions of Theorem \ref{kornsecond}.

\begin{corollary}\label{kornsecond-01} Assume that
the domain $\Omega$ is connected, open bounded with Lipschitz
continuous boundary. Let $D\subset\Omega$ be a subdomain of
$\Omega$, and $Q_{D,RM}$ be the $L^2$ projection from
$[L^2(\Omega)]^d$ onto the space of rigid motions ${\sf RM}(D)$.
Then, there exists a constant $C$ such that
\begin{equation}\label{EQ:KornIneq:2nd-001}
\|\bv\|_{1} \leq C (\|\varepsilon(\bv)\|_{0} +
\|Q_{D,RM}\bv\|_{0,D}),
\end{equation}
for all $\bv\in [H^1(\Omega)]^d$.
\end{corollary}

\begin{proof} Define $\Phi(\bv) =\|Q_{D,RM}\bv\|_{0,D}$ which is clearly a semi-norm in
$[H^1(\Omega)]^d$. The estimate (\ref{EQ:KornIneq:2nd-001}) stems
from Theorem \ref{kornsecond} if the conditions of the theorem are
verified for this semi-norm. To this end, for any $\bv$ satisfying
$\varepsilon(\bv)=0$, we have from Lemma \ref{lemma-rigidmotion}
that $\bv\in {\sf RM}(\Omega)$, and hence $Q_{D,RM}\bv=\bv$ as
$D\subset \Omega$. If, in addition, $\Phi(\bv)=0$, then
$\|\bv\|_{0,D} = \|Q_{D,RM}\bv\|_{0,D}=0$. Hence, $\bv\equiv 0$ in
$\Omega$.
\end{proof}

\begin{corollary}\label{kornsecond-02} Assume that
the domain $\Omega$ is connected, open bounded with Lipschitz
continuous boundary. Let $\Gamma_1\subset \partial\Omega$ be a
nontrivial portion of the boundary $\partial\Omega$ with dimension
$d-1$. For any fixed real number $1\le p < \infty$, there exists a
constant $C$ such that
\begin{equation}\label{EQ:KornIneq:2nd-002}
\|\bv\|_{1} \leq C (\|\varepsilon(\bv)\|_{0} +
\|\bv\|_{L^p(\Gamma_1)}),
\end{equation}
for all $\bv\in [H^1(\Omega)]^d$.
\end{corollary}

\begin{proof} In $[H^1(\Omega)]^d$, we define a semi-norm by
$$
\Phi(\bv):=\left(\int_{\Gamma_1}|\bv|^p
ds\right)^{\frac{1}{p}},\qquad \bv\in [H^1(\Omega)]^d.
$$
We claim that this semi-norm satisfies the conditions of Theorem
\ref{kornsecond}. In fact, if $\bv$ satisfies $\Phi(\bv)=0$, then we
have $\bv=0$ on $\Gamma_1$. Hence $\bv=0$ on $d$-points of a
hyperplane of dimension $d-1$. If, in addition, $\bv\in {\sf
RM}(\Omega)$, then from Lemma \ref{lemma:rigidmotionzero} we must
have $\bv\equiv 0$ on $\Omega$.
\end{proof}

\subsection{Some technical inequalities} For any $\bx_0\in \bbR^d$, denote by $B(\bx_0,
r)$ the $d$-ball centered at $\bx_0$ with radius $r$. The unit
$d$-ball centered at the origin is denoted by $\widehat B$, called
the reference $d$-ball. The reference $d$-ball can be identified
with the $d$-ball $B(\bx_0, r)$ using the following affine map
$$
\bx=F(\widehat\bx): =\bx_0 + r\widehat\bx.
$$
The inverse of the affine map $F$ is given by
$$
\widehat\bx=F^{-1}(\bx): = (\bx-\bx_0)/r.
$$
Any function $\phi=\phi(\bx)$ on $B(\bx_0, r)$ defines a function on
the reference $d$-ball as follows
$$
\widehat\phi(\widehat\bx) = \phi(F(\widehat\bx)),
$$
which shall be denoted as
$$
\widehat\phi(\widehat\bx) = \widehat{\phi(\bx)}.
$$
It is clear that $\phi\in {\sf RM}(B(\bx_0, r))$ if and only if
$\widehat\phi \in {\sf RM}(\widehat B)$.

\begin{lemma}\label{Lemma:Mapping}
Let $Q_{B(\bx_0,r)}$ and $\widehat{Q}_{\widehat B}$ be the $L^2$
projections onto the space of rigid motions ${\sf RM}(B(\bx_0,r))$
and ${\sf RM}(\widehat{B})$, respectively. Then, for any $\bv\in
[L^2(B(\bx_0,r))]^d$ we have
\begin{equation}\label{EQ:L2-for-Mapping}
\widehat{Q_{B(\bx_0,r)} \bv} = \widehat{Q}_{\widehat B}
\widehat{\bv}.
\end{equation}
\end{lemma}
\begin{proof}
From the definition of the $L^2$ projection, we have for any
$\phi\in {\sf RM}(B(\bx_0,r))$
$$
\int_{B(\bx_0,r)} Q_{B(\bx_0,r)} \bv \ \phi dx =  \int_{B(\bx_0,r)}
 \bv \ \phi dx = r^d \int_{\widehat B} \widehat\bv \ \widehat\phi
 d\widehat{x} = r^d \int_{\widehat B} \widehat Q_{\widehat B}\widehat\bv \ \widehat\phi
 d\widehat{x}.
$$
Also by changing the domain from $B(\bx_0,r)$ to $\widehat B$,
$$
\int_{B(\bx_0,r)} Q_{B(\bx_0,r)} \bv \ \phi dx = r^d \int_{\widehat
B}\widehat{Q_{B(\bx_0,r)} \bv} \ \widehat\phi
 d\widehat{x}.
$$
It follows that
$$
\int_{\widehat B}\widehat{Q_{B(\bx_0,r)} \bv} \ \widehat\phi
 d\widehat{x} = \int_{\widehat B} \widehat Q_{\widehat B}\widehat\bv \ \widehat\phi
 d\widehat{x},
$$
which leads to (\ref{EQ:L2-for-Mapping}).
\end{proof}

\begin{lemma}\label{Lemma:July29-001} There exists a constant $C$ such
that, for any $\bw\in [H^1(B(\bx_0,r))]^d$, we have
\begin{equation*}
\|\bw-Q_{B(\bx_0,r)}\bw\|_{0, B(\bx_0,r)} \leq C
r\|\varepsilon(\bw)\|_{0, B(\bx_0,r)}.
\end{equation*}
\end{lemma}

\begin{proof} Let $\bw^\bot = \bw - Q_{B(\bx_0,r)}\bw$.
It follows from Lemma \ref{Lemma:Mapping} that
$$
\widehat{\bw^\bot} = \widehat{\bw} - \widehat{Q}_{\widehat
B}\widehat{\bw}.
$$
It is also easy to see the following identities
\begin{equation}\label{Eq:July29:800}
\widehat{Q}_{\widehat B} \widehat{\bw^\bot} =0, \quad
\varepsilon(\widehat{\bw^\bot}) =\varepsilon(\widehat{\bw}).
\end{equation}
By mapping to the reference $d$-ball, we have
$$
\|\bw^\bot\|^2_{0,B(\bx_0,r)} = r^d
\|\widehat{\bw^\bot}\|_{0,\widehat B}^2.
$$
Using the second Korn's inequality (\ref{EQ:KornIneq:2nd-001}) with
$\Omega=\widehat B$ and $D=\widehat B$, we obtain
\begin{eqnarray*}
\|\widehat{\bw^\bot}\|_{0,\widehat B}^2 &\leq &C(
\|\varepsilon(\widehat{\bw^\bot})\|_{0,\widehat B}^2 +
\|\widehat{Q}_{\widehat B} \widehat{\bw^\bot}\|_{0,\widehat B}^2)\\
& \leq & C \|\varepsilon(\widehat{\bw})\|_{0,\widehat B}^2\\
&\leq & C r^{2-d} \|\varepsilon(\bw)\|_{0,B(\bx_0,r)}^2,
\end{eqnarray*}
where we have used (\ref{Eq:July29:800}) in the second line. Thus,
one has
$$
\|\bw - Q_{B(\bx_0,r)}\bw\|^2_{0,B(\bx_0,r)} =
\|\bw^\bot\|^2_{0,B(\bx_0,r)} \leq Cr^{2}
\|\varepsilon(\bw)\|_{0,B(\bx_0,r)}^2.
$$
This completes the proof of the lemma.
\end{proof}

\begin{lemma}\label{korninqua} Let $\T_h$ be a shape regular finite element
partition of $\Omega$. There exists a constant $C$ independent of
$T\in \T_h$ such that
\begin{equation}\label{EQ:July29:815}
\|\bv_0-Q_b\bv_0\|_{\partial T}^2 \leq
Ch_T\|\varepsilon(\bv_0)\|_T^2,\qquad \forall\ \bv_0\in [P_k(T)]^d.
\end{equation}
\end{lemma}

\begin{proof} From the shape regularity assumption, there exists a
$d$-ball $B(\bx_0,r)\subset T$ for which the radius $r$ is
proportional to $h_T$; i.e., $r= \lambda_0 h_T$ with a constant
$\lambda_0$ bounded away from $0$. For any $\bv_0\in [P_k(T)]^d$,
consider the projection $Q_{B(\bx_0,r)}\bv_0$ which is naturally
extended to $T$. Since $(Q_{B(\bx_0,r)}\bv_0)|_\pT$ belongs to the
finite element space on $\pT$ and $Q_b$ is the $L^2$ projection onto
this finite element space, then
\begin{equation}\label{EQ:July29:810}
\|\bv_0-Q_b\bv_0\|_{\partial T} \leq
\|\bv_0-Q_{B(\bx_0,r)}\bv_0\|_{\partial T}.
\end{equation}
Using the trace and inverse inequality (\ref{A4})-(\ref{Aa}) we have
\begin{eqnarray*}
\|\bv_0-Q_{B(\bx_0,r)}\bv_0\|_{\partial T}^2 &\leq &C(
h_T^{-1}\|\bv_0- Q_{B(\bx_0,r)}\bv_0\|_{T}^2+h_T
\|\nabla(\bv_0-Q_{B(\bx_0,r)}\bv_0)\|_{T}^2)\\
&\leq & Ch_T^{-1}\|\bv_0- Q_{B(\bx_0,r)}\bv_0\|_{T}^2\\
&\leq & C h_T^{-1}\|\bv_0- Q_{B(\bx_0,r)}\bv_0\|_{B(\bx_0,r)}^2,
\end{eqnarray*}
where we have applied the domain inverse inequality (see the
Appendix in \cite{wy1202}) in the third line. Now applying Lemma
\ref{Lemma:July29-001} to the term $\|\bv_0-
Q_{B(\bx_0,r)}\bv_0\|_{B(\bx_0,r)}^2 $ we obtain
\begin{equation*}
\|\bv_0- Q_{B(\bx_0,r)}\bv_0\|_{\pT}^2 \leq C h_T
\lambda_0^2\|\varepsilon(\bv_0)\|_{B(\bx_0,r)}^2\leq C h_T
\|\varepsilon(\bv_0)\|_{T}^2.
\end{equation*}
Combining (\ref{EQ:July29:810}) with the above estimate gives
(\ref{EQ:July29:815}). This completes the proof.
\end{proof}

\begin{lemma}\label{lem8.2} Assume that the finite
element partition ${\cal T}_h$ of $\Omega$ is shape regular, and
the finite element space $V_h$ is constructed as in Section
\ref{Section:fem-algorithms} with $k\ge 1$. Let $1 \leq r \leq k$.
Then, for any $\bw\in [H^{r+1} (\Omega)]^d$, $\rho \in H^r
(\Omega)$ and $\bv \in V_h$, the following estimates hold true
\begin{align}\label{A7}
|s(Q_h\bw,\bv)|\leq &
Ch^r\|\bw\|_{r+1}\3bar\bv\3bar,\\
|\ell_\bw(\bv)|\leq&
Ch^r\|\bw\|_{r+1}\3bar\bv\3bar,\label{A8}\\
|\theta_\rho(\bv)|\leq &
Ch^r\|\rho\|_r\3bar\bv\3bar,\label{A9}
\end{align}
where $\ell_\textbf{\bw}(\cdot)$ and $\theta_\rho(\cdot)$ are given
as in (\ref{l}) and (\ref{theta}).
\end{lemma}

\begin{proof}
To derive (\ref{A7}), we use (\ref{EQ:stabilizer}), the
Cauchy-Schwarz inequality, and the estimates (\ref{A4}) and
(\ref{A1}) to obtain
\begin{equation*}
\begin{split}
|s(Q_h\bw,\bv)|=&\Big|\sum_{T\in {\cal
T}_h}h_T^{-1}\langle Q_bQ_0\bw-Q_b\bw, Q_b\bv_0-
\bv_b\rangle_{\partial
T}\Big|\\
=&\Big|\sum_{T\in {\cal T}_h}h_T^{-1}\langle
 Q_0\bw - \bw,Q_b\bv_0-
\bv_b\rangle_{\partial
T}\Big|\\
\leq & \Big(\sum_{T\in {\cal T}_h}h_T^{-1}\|Q_0\bw-
\bw\|_{\partial T}^2\Big)^{\frac{1}{2}}\Big(\sum_{T\in {\cal
T}_h}h_T^{-1}\|Q_b\bv_0- \bv_b\|_{\partial
T}^2\Big)^{\frac{1}{2}}\\
\leq & C\Big(\sum_{T\in {\cal T}_h}h_T^{-2}\|Q_0\bw- \bw\|_{T}^2 +
\|\nabla(Q_0\bw- \bw)\|_{T}^2\Big)^{\frac{1}{2}} \3bar \bv \3bar\\
\leq & Ch^r\|\bw\|_{r+1}\3bar \bv \3bar.
\end{split}
\end{equation*}

To prove (\ref{A8}), we use the Cauchy-Schwarz inequality, the
estimates (\ref{A4}) and (\ref{A2}) to get
\begin{equation}\label{korn}
\begin{split}
|\ell_\bw(\bv)|=&\left|2\mu \sum_{T\in{\cal
T}_h}\langle\bv_0-\bv_b,( \varepsilon(\bw)-\textbf{Q}_h\varepsilon(
\bw))\bn\rangle_{\partial T}\right|\\
\leq & C \Big(\sum_{T\in{\cal T}_h}h_T^{-1}\|
 \bv_0-\bv_b\|^2_{\partial T}\Big)^{ \frac{1}{2}}
\Big(\sum_{T\in{\cal T}_h}h_T \|
\varepsilon(\bw)-\textbf{Q}_h\varepsilon(\bw)\|^2_{\partial
T}\Big)^{ \frac{1}{2}}\\
\leq & C h^r\|\bw\|_{r+1}\Big(\sum_{T\in{\cal T}_h}h_T^{-1}\|
 \bv_0-\bv_b\|^2_{\partial T}\Big)^{ \frac{1}{2}}.
\end{split}
\end{equation}
For the term  $ \sum_{T\in{\cal T}_h}h_T^{-1}\|
 \bv_0-\bv_b\|^2_{\partial T} $, from Lemma \ref{korninqua} we have
\begin{equation}\label{korn1}
\begin{split}
 & \sum_{T\in{\cal T}_h}h_T^{-1}\|
 \bv_0-\bv_b\|^2_{\partial T}\\
\leq &  2 \sum_{T\in{\cal T}_h}h_T^{-1}\|
 \bv_0-Q_b\bv_0\|^2_{\partial T}+ 2\sum_{T\in{\cal T}_h}h_T^{-1}\|
 Q_b\bv_0-\bv_b\|^2_{\partial T}\\
\leq & C \sum_{T\in{\cal T}_h} \| \varepsilon( \bv_0)\|^2_{  T}+2
\sum_{T\in{\cal T}_h}h_T^{-1}\|
 Q_b\bv_0-\bv_b\|^2_{\partial T}\\
\leq & C \3bar \bv\3bar^2 .
\end{split}
\end{equation}
Substituting (\ref{korn1}) into (\ref{korn}) yields
\begin{equation*}
 |\ell_\bw(\bv)|
\leq  C h^r\|\bw\|_{r+1}\3bar \textbf{v}\3bar.
 \end{equation*}

As to (\ref{A9}), we use the Cauchy-Schwarz inequality, the
estimates (\ref{A3}) and (\ref{korn1}) to obtain
\begin{equation*}
\begin{split}
|\theta_\rho(\textbf{v})|&=\left|\sum_{T\in{\cal T}_h}\langle
 \textbf{v}_0-\textbf{v}_b  ,
(\rho- {\cal Q}_h\rho) \textbf{n}\rangle_{\partial T}\right|\\
&\leq  \Big(\sum_{T\in{\cal T}_h}h_T^{-1}\|
 \textbf{v}_0-\textbf{v}_b\|^2_{\partial T}\Big)^{ \frac{1}{2}}
\Big(\sum_{T\in{\cal T}_h}h_T \|  (\rho- {\cal Q}_h\rho)
\textbf{n} \|^2_{\partial
T}\Big)^{ \frac{1}{2}}\\
 &\leq  C h^r\| \rho\|_{r}\3bar\textbf{v}\3bar.
\end{split}
\end{equation*}
This completes the proof of the lemma.
\end{proof}

\section{Numerical Results}
In this section, we shall report some numerical results for the weak
Galerkin finite element scheme proposed and analyzed in previous
sections. In our numerical investigation, we considered the linear
elasticity equation (\ref{primal_model}) in the two dimensional
square domain $\Omega =(0,1)^2$. The weak Galerkin approximations
are obtained by using the lowest order finite elements. More
precisely, the finite element space is given either by
\begin{equation}\label{EQ:August01:001}
{V}_h=\{\bv=\{\bv_0,\bv_b\}:\ \bv_0\in[P_1(T)]^2,\bv_b\in
P_{RM}(e),\ e\subset\partial T,\ T\in\T_h\}
\end{equation}
or by
\begin{equation}\label{EQ:August01:002}
\bar{V}_h=\{\bv=\{\bv_0,\bv_b\}:\ \bv_0\in[P_1(T)]^2,\bv_b\in
P_{1}(e),\ e\subset\partial T,\ T\in\T_h\}.
\end{equation}
Since $P_{RM}(e)\subset P_1(e)$ for each edge $e\subset \pT$, then
$V_h$ is clearly a subspace of $\bar{V}_h$. It can be seen that all
the theoretical results developed for $V_h$ can be extended to
$\bar{V}_h$ without any difficulty.

The discrete weak gradient and the discrete weak divergence are
computed on each element $T\in\T_h$ according to their definition.
More precisely, for any $\bv\in {V}_h$ or $\bv\in \bar{V}_h$, the
discrete weak gradient $\nabla_w\bv$ on the element $T$ is given by
\begin{eqnarray}\label{EQ:August01:003}
(\nabla_w\bv,\varphi)_T&=&\langle\bv_b,\varphi\cdot\bn\rangle_{\partial
T},\qquad \forall \ \varphi\in[P_0(T)]^{2\times2}.
\end{eqnarray}
The discrete weak divergence $\nabla_w\cdot\bv$ on $T$ is given by
\begin{eqnarray}\label{EQ:August01:004}
(\nabla_w\cdot\bv,\psi)_T&=&\langle\bv_b,\psi\bn\rangle_{\partial
T},\qquad \forall \ \psi\in P_0(T).
\end{eqnarray}

Denote by $\bu_h=\{\bu_0,\bu_b\}$ and $\bu$ the solution to the weak
Galerkin formulation (\ref{WGA_primal}) and the original equation
(\ref{model-weak}), respectively. Define the error by $\be_h = Q_h
\bu - \bu_h = \{\be_0,\,\be_b\}$ where $Q_h \bu$ is the $L^2$
projection of the exact solution $\bu$ in $V_h$ or $\bar{V}_h$, as
appropriate. The error for the weak Galerkin finite element solution
is computed in three norms defined as follows
\begin{eqnarray*}
\3bar\be_h\3bar_{*}^2&=&\sum_{T\in\mathcal{T}_h}\left\{2\mu\int_{T}|\varepsilon_w(\be_h)|^2dT+\lambda\int_{T}|\nabla_w\cdot\be_h|^2dT
+h_T^{-1}\int_{\partial T}|Q_b\be_0-\be_b|^2ds\right\},
\\
\|\be_0\|^2&=&\sum_{T\in\mathcal{T}_h}\int_T|\be_0|^2dT,
\\
\|\be_b\|^2&=&\sum_{T\in\mathcal{T}_h}h_T\int_{\partial
T}|\be_b|^2ds.
\end{eqnarray*}
It can be seen that $\3bar\be_h\3bar_{*}$ is an $H^1$-like norm
for the error function, and $\|\be_0\|$ and $\|\be_b\|$ are
typical $L^2$ norms for the error in the interior and on the
boundary of each element.

\subsection{Test Problem 1}
Consider the elasticity equation (\ref{primal_model}) in the square
domain $\Omega =(0,1)^2$ which is partitioned into uniform
triangular mesh $\T_h$ with mesh size $h$. The right-hand side
function $\bbf$ is chosen so that the exact solution is given by
\begin{eqnarray*}
\bu=\left(\begin{array}{ccc}
\sin(x) \sin(y) \\
1
\end{array}\right).
\end{eqnarray*}

Table \ref{tab:ex1:001} illustrates the computational result when
the rigid motions are employed on the boundary of each element;
i.e., the space $V_h$. Table \ref{tab:ex1:002} shows the result when
linear functions are used on the boundary of each element; i.e., the
space $\bar V_h$. Theoretically, these two methods have the same
order of convergence which is confirmed by these two tables. Note
that the WG method with $V_h$ has less number of degrees of freedom
than that of $\bar V_h$.

A numerical scheme is said to be convergent with order $\alpha$ if
the error decreases proportionally to $h^\alpha$, where $h$ is the
mesh parameter. From Tables \ref{tab:ex1:001} and \ref{tab:ex1:002}
we see that the convergence of the weak Galerkin finite element
scheme in the $L^2$-norm is of order $2$ and that in the $H^1$-norm
is of order $1$. The numerical results are in consistency with the
theoretical prediction.

\begin{table}[H]
  \caption{Problem 1: $\lambda=1$, $\mu=0.5$,
  and WG with $V_h$ (the rigid motion space on element boundary).}
\label{tab:ex1:001}
\begin{tabular}{||c|c|c|c|c|c|c||}
\hline
1/h&$\|\bu_0-Q_0\bu\|$&order&$\|\bu_b-Q_b\bu\|$&order&$\3bar\bu_h-Q_h\bu\3bar_*$&order\\
\hline
2&0.0750&--&0.0424&--&0.3103&--\\
4&0.0192&1.97&0.0115&1.88&0.1566&0.99\\
8&0.0049&1.98&0.0031&1.87&0.0787&0.99\\
16&0.0012&1.99&0.0008&1.93&0.0394&1.00\\
32&0.0003&2.00&0.0002&1.97&0.0197&1.00\\
 \hline
\end{tabular}
\end{table}

\begin{table}[H]
  \caption{Problem 1: $\lambda=1$, $\mu=0.5$,
  and WG with $\bar V_h$ (linear functions on element boundary).}
\label{tab:ex1:002}
\begin{tabular}{||c|c|c|c|c|c|c||}
\hline
1/h&$\|\bu_h-Q_0\bu\|$&order&$\|\bu_b-Q_b\bu\|$&order&$\3bar\bu_h-Q_h\bu\3bar$&order\\
\hline
2&0.0743&--&0.0424&--&0.3082&--\\
4&0.0190&1.96&0.0113&1.90&0.1555&0.99\\
8&0.0048&1.98&0.0031&1.88&0.0782&0.99\\
16&0.0012&1.99&0.0008&1.93&0.0392&1.00\\
32&0.0003&2.00&0.0002&1.97&0.0196&1.00\\
 \hline
\end{tabular}
\end{table}

\subsection{Test Problem 2} In the second test, the linear elasticity equation
(\ref{primal_model}) is also defined on the unit square domain
$\Omega =(0,1)^2$, but the exact solution is given as follows
\begin{eqnarray*}
\bu=\left(\begin{array}{ccc}
(x+y)^2 \\
(x-y)^2
\end{array}\right).
\end{eqnarray*}
The right-hand side function $\bbf$ is computed to match the exact
solution. The numerical results, as shown in Tables
\ref{tab:ex2:001}-\ref{tab:ex2:002}, are again based on the uniform
partition $\mathcal{T}_h$ for this domain. The numerical results
confirm the theoretical prediction developed in previous sections.

\begin{table}[H]
\caption{Problem 2: $\lambda=1$, $\mu=0.5$,
  and WG with $V_h$ (the rigid motion space on element boundary).}
\label{tab:ex2:001}
\begin{tabular}{||c|c|c|c|c|c|c||}
\hline
1/h&$\|\bu_0-Q_0\bu\|$&order&$\|\bu_b-Q_b\bu\|$&order&$\3bar\bu_h-Q_h\bu\3bar_*$&order\\
\hline
2&0.4334&--&0.2089&--&1.8410&--\\
4&0.1095&1.98&0.0528&1.98&0.9264&0.99\\
8&0.0275&2.00&0.0133&1.99&0.4644&1.00\\
16&0.0069&2.00&0.0034&2.00&0.2324&1.00\\
32&0.0017&2.00&0.0008&2.00&0.1162&1.00\\
 \hline
\end{tabular}
\end{table}

\medskip

\begin{table}[H]
  \caption{Problem 2: $\lambda=1$, $\mu=0.5$, and WG with $\bar V_h$ (linear functions on element boundary).}
\label{tab:ex2:002}
\begin{tabular}{||c|c|c|c|c|c|c||}
\hline
1/h&$\|\bu_h-Q_0\bu\|$&order&$\|\bu_b-Q_b\bu\|$&order&$\3bar\bu_h-Q_h\bu\3bar$&order\\
\hline
2&0.4310&--&0.2009&--&1.8355&--\\
4&0.1088&1.99&0.0502&2.00&0.9233&0.99\\
8&0.0273&2.00&0.0126&2.00&0.4628&1.00\\
16&0.0068&2.00&0.0032&1.99&0.2316&1.00\\
32&0.0017&2.00&0.0008&2.00&0.1158&1.00\\
 \hline
\end{tabular}
\end{table}

\subsection{Locking-Free Tests}
In the locking-free investigation, the linear elasticity equation
(\ref{primal_model}) has exact solutions given by
\begin{eqnarray*}
\bu=\left(\begin{array}{ccc}
\sin(x)\sin(y) \\
\cos(x)\cos(y)
\end{array}\right)
+\lambda^{-1}\left(\begin{array}{ccc}
x \\
y
\end{array}\right).
\end{eqnarray*}
The right-hand side function $\bbf$ is computed to match the exact
solution (note that it is $\lambda$-dependent). The numerical
results are based on the same uniform partition $\mathcal{T}_h$ for
the unit square domain. The results, as shown in Tables
\ref{Locking-Free:RM:01}-\ref{Locking-Free:P1:04}, clearly indicate
a locking-free convergence for the weak Galerkin finite element
method in various norms, which is consistent with theory.

\begin{table}[H]
\caption{WG with $V_h$ based on the element of
$\{P_1(T)/P_{RM}(e)\}$, $\mu=0.5$, and $\lambda=1$.}
\label{Locking-Free:RM:01}
\begin{tabular}{|c|c|c|c|c|c|c|}
\hline
1/h&$\|\bu_h-Q_0\bu\|$&order&$\|\bu_b-Q_b\bu\|$&order&$\3bar\bu_h-Q_h\bu\3bar$&order\\
\hline
2&0.0352&--&0.0331&--&0.1544&--\\
4&0.0097&1.86&0.0120&1.46&0.0834&0.89\\
8&0.0026&1.91&0.0037&1.68&0.0433&0.94\\
16&0.0007&1.96&0.0010&1.87&0.0220&0.98\\
32&0.0002&1.98&0.0003&1.96&0.0110&0.99\\
 \hline
\end{tabular}
\end{table}

\begin{table}[H]
\caption{WG with $V_h$ based on the element of
$\{P_1(T)/P_{RM}(e)\}$, $\mu=0.5$, and $\lambda=100$.}
\label{Locking-Free:RM:02}
\begin{tabular}{|c|c|c|c|c|c|c|}
\hline
1/h&$\|\bu_h-Q_0\bu\|$&order&$\|\bu_b-Q_b\bu\|$&order&$\3bar\bu_h-Q_h\bu\3bar$&order\\
\hline
2&0.0344&--&0.0291&--&0.1449&--\\
4&0.0100&1.79&0.0113&1.36&0.0774&0.90\\
8&0.0028&1.82&0.0038&1.59&0.0403&0.94\\
16&0.0008&1.91&0.0011&1.81&0.0205&0.97\\
32&0.0002&1.96&0.0003&1.93&0.0103&0.99\\
 \hline
\end{tabular}
\end{table}

\begin{table}[H]
\caption{WG with $V_h$ based on the element of
$\{P_1(T)/P_{RM}(e)\}$, $\mu=0.5$, and $\lambda=10,000$.}
\label{Locking-Free:RM:03}
\begin{tabular}{|c|c|c|c|c|c|c|}
\hline
1/h&$\|\bu_h-Q_0\bu\|$&order&$\|\bu_b-Q_b\bu\|$&order&$\3bar\bu_h-Q_h\bu\3bar$&order\\
\hline
2&0.0344&--&0.0290&--&0.1447&--\\
4&0.0100&1.79&0.0113&1.36&0.0773&0.90\\
8&0.0028&1.82&0.0038&1.59&0.0403&0.94\\
16&0.0008&1.90&0.0011&1.81&0.0205&0.97\\
32&0.0002&1.96&0.0003&1.93&0.0103&0.99\\
 \hline
\end{tabular}
\end{table}

\begin{table}[H]
\caption{WG with $V_h$ based on the element of
$\{P_1(T)/P_{RM}(e)\}$, $\mu=0.5$, and $\lambda=1,000,000$.}
\label{Locking-Free:RM:04}
\begin{tabular}{|c|c|c|c|c|c|c|}
\hline
1/h&$\|\bu_h-Q_0\bu\|$&order&$\|\bu_b-Q_b\bu\|$&order&$\3bar\bu_h-Q_h\bu\3bar$&order\\
\hline
2&0.0344&--&0.0290&--&0.1447&--\\
4&0.0100&1.79&0.0113&1.36&0.0773&0.90\\
8&0.0028&1.82&0.0038&1.59&0.0403&0.94\\
16&0.0008&1.90&0.0011&1.81&0.0205&0.97\\
32&0.0002&1.96&0.0003&1.93&0.0103&0.99\\
 \hline
\end{tabular}
\end{table}

\begin{table}[H]
\caption{WG with $\bar V_h$ based on the element of
$\{P_1(T)/P_{1}(e)\}$, $\mu=0.5$, and
$\lambda=1$.}\label{Locking-Free:P1:01}
\begin{tabular}{|c|c|c|c|c|c|c|}
\hline
1/h&$\|\bu_h-Q_0\bu\|$&order&$\|\bu_b-Q_b\bu\|$&order&$\3bar\bu_h-Q_h\bu\3bar$&order\\
\hline
2&0.0341&--&0.0313&--&0.1518&--\\
4&0.0093&1.87&0.0115&1.45&0.0816&0.90\\
8&0.0025&1.91&0.0036&1.67&0.0424&0.95\\
16&0.0006&1.96&0.0010&1.86&0.0215&0.98\\
32&0.0002&1.98&0.0003&1.95&0.0108&0.99\\
 \hline
\end{tabular}
\end{table}

\begin{table}[H]
\caption{WG with $\bar V_h$ based on the element of
$\{P_1(T)/P_{1}(e)\}$, $\mu=0.5$, and
$\lambda=100$.}\label{Locking-Free:P1:02}
\begin{tabular}{|c|c|c|c|c|c|c|}
\hline
1/h&$\|\bu_h-Q_0\bu\|$&order&$\|\bu_b-Q_b\bu\|$&order&$\3bar\bu_h-Q_h\bu\3bar$&order\\
\hline
2&0.0340&--&0.0281&--&0.1441&--\\
4&0.0098&1.79&0.0111&1.34&0.0769&0.91\\
8&0.0028&1.82&0.0037&1.58&0.0400&0.94\\
16&0.0007&1.90&0.0011&1.81&0.0204&0.97\\
32&0.0002&1.96&0.0003&1.93&0.0102&0.99\\
 \hline
\end{tabular}
\end{table}

\begin{table}[H]
\caption{WG with $\bar V_h$ based on the element of
$\{P_1(T)/P_{1}(e)\}$, $\mu=0.5$, and
$\lambda=10,000$.}\label{Locking-Free:P1:03}
\begin{tabular}{|c|c|c|c|c|c|c|}
\hline
1/h&$\|\bu_h-Q_0\bu\|$&order&$\|\bu_b-Q_b\bu\|$&order&$\3bar\bu_h-Q_h\bu\3bar$&order\\
\hline
2&0.0340&--&0.0280&--&0.1439&--\\
4&0.0098&1.79&0.0111&1.34&0.0768&0.91\\
8&0.0028&1.82&0.0037&1.58&0.0400&0.94\\
16&0.0007&1.90&0.0011&1.81&0.0203&0.97\\
32&0.0002&1.96&0.0003&1.93&0.0102&0.99\\
 \hline
\end{tabular}
\end{table}

\begin{table}[H]
\caption{WG with $\bar V_h$ based on the element of
$\{P_1(T)/P_{1}(e)\}$, $\mu=0.5$, and
$\lambda=1,000,000$.}\label{Locking-Free:P1:04}
\begin{tabular}{|c|c|c|c|c|c|c|}
\hline
1/h&$\|\bu_h-Q_0\bu\|$&order&$\|\bu_b-Q_b\bu\|$&order&$\3bar\bu_h-Q_h\bu\3bar$&order\\
\hline
2&0.0340&--&0.0280&--&0.1439&--\\
4&0.0098&1.79&0.0111&1.34&0.0768&0.91\\
8&0.0028&1.82&0.0037&1.58&0.0400&0.94\\
16&0.0007&1.90&0.0011&1.81&0.0203&0.97\\
32&0.0002&1.96&0.0003&1.93&0.0102&0.99\\
 \hline
\end{tabular}
\end{table}


\begin{thebibliography}{99}

\bibitem{abd1984}  {\sc D. N. Arnold, F. Brezzi and J. Douglas},
 {\em PEERS: a new mixed finite element for plane
elasticity}, Japan J. Appl. Math., 1 (1984), no. 2, 347-367.

\bibitem{babuska-suri} {\sc I. Babu\v{s}ka and M. Suri},
{\em Locking effects in the finite element approximation of
elasticity problems}, Numer. Math. 62 (1992), Issue 1, 439-463.

\bibitem{b2003}  {\sc S. Brenner},
 {\em Korn's inequalities for piecewise $H^1$ vector fields},
Math. Comp. 73 (2003), 1067-1087.

\bibitem{Brezzi-Fortin} {\sc F. Brezzi and M. Fortin},
{\em Mixed and Hybrid Finite Element Methods}, Springer-Verlag, New
York, 1991.

\bibitem{ciarlet-fem}
{\sc P. G. Ciarlet}, \textit{The Finite Element Method for Elliptic
Problems}, North-Holland, 1978.

\bibitem{Ciarlet} {\sc P. G. Ciarlet},
{\em Mathematical Elasticity, I. Three-Dimensional Elasticity},
North-Hollad, Amsterdam, 1988.

\bibitem{Dauge} {\sc M. Dauge},
{\em Elliptic Boundary Value Problems on Corner Domains - Smoothness
and Asymptotics of Solutions}, Volume 1341 of {\em Lecture Notes in
mathematics}. Springer, Berlin, 1988.

\bibitem{Daniele.01}
{\sc D. A. Di Pietro and S. Nicaise}, {\em A locking-free
discontinuous Galerkin method for linear elasticity in locally
nearly incompressible heterogeneous media}, J. Applied Numerical
Mathematics, 63 (2013), 105-116.

\bibitem{Girault} {\sc V. Girault and P.-A. Raviart},
{\em Finite Element Methods for Navier-Stokes Equations: Theory and
Algorithms}, Springer-Verlag, 1984.

\bibitem{Grisvard} {\sc P. Grisvard},
{\em Elliptic Problems in Nonsmooth Domains}, Volume 24 of {\em
Monographs and Studies in Mathematics}. Pitman,
Boston-London-Melbourne, 1985.

\bibitem{Duvaut} {\sc G. Guvaut and J. L. Lions},
{\em Inequalities in Mechanics and Physics}, Springer-Verlag, 1976.

\bibitem{ff1997}  {\sc M. Farhloul and M. Fortin},
 {\em Dual hybrid methods for the elasticity and the Stokes problems:
a unified approach}, Numer. Math., 76 (1997), 419-440.

\bibitem{larson}  {\sc P. Hansbo and M. G. Larson},
 {\em Discontinuous Galerkin methods for incompressible and nearly
 incompressible elasticity by Nitsche's methods},
 Comput. Methods Appl. Mech. Engrg., vol. 191 (2002),  no. 17-18, 1895-1908.

\bibitem{mwy2}
{\sc L. Mu, J. Wang, and X. Ye}, {\em A weak Galerkin finite element
methods with polynomial reduction}, arXiv:1304.6481,  J. Comp. and
Appl. Math., 285 (2015), 45-58. doi:10.1016/j.cam.2015.02.001.

\bibitem{mwy}  {\sc L. Mu, J. Wang  and X. Ye},
{\em Weak Galerkin finite element methods for the biharmonic
equation on polytopal meshes}, Numerical Methods for Partial
Differential Equations, 30 (2014), 1003-1029.

\bibitem{Nitsche} {\sc J.A. Nitsche},
{\em On Korn's second inequality}, RAIRO, Analyse Numerique, 15
(1981), no. 3, 237-248.

\bibitem{scott-vogelius} {\sc L.R. Scott and M. Vogelius},
{\em Norm estimates for a maximal right inverse of the divergence
operator in spaces of piecewise polynomials}, RAIRO Math. Modeling
Num. Anal. 19 (1985), 111-143.

\bibitem{vogelius} {\sc M. Vogelius},
{\em An analysis of the $p$-version of the finite element method for
nearly incompressible materials}, Numerische Mathematik, 41 (1983),
no. 1, 39-53.

\bibitem{wwbi}  {\sc C. Wang and J. Wang},
{\em An efficient numerical scheme for the biharmonic equation by
weak Galerkin finite element methods on polygonal or polyhedral
meshes}, arXiv:1303.0927v1. Computers and Mathematics with
Applications, 68 (2014), 2314-2330. DOI:10.1016
/j.camwa.2014.03.021.

\bibitem{ww-survey}
{\sc J. Wang and C. Wang}, {\em Weak Galerkin finite element methods
for elliptic PDEs (in Chinese)}, Sci Sin Math, 45 (2015), 1061-1092,
doi:10.1360/N012014-00233.

\bibitem{wy1202}  {\sc J. Wang and X. Ye},
 {\em A weak Galerkin mixed finite element method for second-order elliptic problems},
 arXiv:1202.3655v1. Math. Comp., 83 (2014), 2101-2126.

\bibitem{wy1302} {\sc  J. Wang and X. Ye},
 {\em A weak Galerkin finite element method for
 the Stokes equations}, arXiv:1302.2707v1. Advances in Computational Mathematics,
 May, 2015. DOI 10.1007/s10444-015-9415-2.

\bibitem{wysec2}  {\sc J. Wang and X. Ye},
 {\em A weak Galerkin finite element method for
second-order elliptic problems}, J. Comp. and Appl. Math, 241(2013),
103-115.

\bibitem{Wihler}  {\sc T. P. Wihler},
{\em Locking-free adaptive discontinuous Galerkin FEM for linear
elasticity problems}, Math. Comp., 75(2006), no. 255, 1087-1102.

\end{thebibliography}
\end{document}